\documentclass[11pt]{article}
\usepackage{mathrsfs}
\usepackage{amsmath,amsthm,amssymb}
\usepackage{amsfonts}
\usepackage{color,xcolor}
\usepackage{graphicx}
\usepackage{geometry}
\usepackage{manfnt}
\usepackage[pagewise]{lineno}
\usepackage[colorlinks,citecolor=blue,urlcolor=blue]{hyperref}
\allowbreak
\newtheorem{theorem}{Theorem}[section]

\newtheorem{lemma}[theorem]{Lemma}

\newtheorem{remark}[theorem]{Remark}

\begin{document}
\title{\bf The Euler--Maruyama method for SDEs with low-regularity drift}

\author{Jinlong Wei$^a$, Junhao Hu$^b$, Guangying Lv$^c$ and Chenggui Yuan$^d$\thanks{Corresponding author.}
\\{\small \it $^a$School of Statistics and Mathematics, Zhongnan University of Economics}\\
{\small \it  and Law, Wuhan 430073, China}  \\ {\small \tt  weijinlong.hust@gmail.com}
\\ {\small \it $^b$School of Mathematics and Statistics, South-Central Minzu University} \\ {\small \it Wuhan 430074, China} \\ {\small \tt junhaohu74@163.com}
\\ {\small \it $^c$College of Mathematics and Statistics,
Nanjing University of Information}
\\{\small \it  Science and Technology, Nanjing 210044, China}
\\ {\small \tt  gylvmaths@126.com} \\ {\small \it $^d$Department of Mathematics, Swansea University, Bay Campus} \\ {\small \it Swansea SA1 8EN, United Kingdom} \\ {\small \tt C.Yuan@swansea.ac.uk}}

\date{}

\maketitle
\noindent{\hrulefill}
\vskip1mm\noindent{\bf Abstract} We study the strong $L^p$-convergence rates of the Euler--Maruyama method for stochastic differential equations  driven by Brownian motion with low-regularity drift coefficients.
Specifically, the drift is assumed to be in the Lebesgue--H\"{o}lder spaces $L^q([0,T];
{\mathcal C}_b^\alpha({\mathbb R}^d))$ with $\alpha\in(0,1)$ and $q\in (2/(1+\alpha),\infty]$. For every $p\geq 2$, by using stochastic sewing and/or the It\^{o}--Tanaka trick, we obtain the  $L^p$-convergence rates: $(1+\alpha)/2$ for $q\in [2,\infty]$ and $(1-1/q)$ for $q\in (2/(1+\alpha),2)$. Moreover, we prove that the unique strong solution can be constructed via the Picard iteration.


\vskip1mm\noindent{\hrulefill}

\section{Introduction}\label{sec1}\setcounter{equation}{0}
Let $T>0$ be a given real number. We are concerned with numerical method and aim to obtain its strong convergence rates for the following stochastic differential equation (SDE for short) in ${\mathbb R}^d$
\begin{eqnarray}\label{1.1}
dX_t=b(t,X_t)dt+dB_t, \ t\in (0,T], \
X_t|_{t=0}=x\in{\mathbb R}^d,
\end{eqnarray}
where $\{B_t\}_{t\in [0,T]}$ is a $d$-dimensional standard Brownian motion starting from $0$ defined on a given stochastic basis ($\Omega, {\mathcal F},\{{\mathcal F}_t\}_{t\in [0,T]},{\mathbb P}$) and $b: [0,T]\times{\mathbb R}^d\rightarrow{\mathbb R}^d$ is Borel measurable.

When the drift coefficient $b$ is Lipschitz continuous in spatial variables, the strong well-posedness for (\ref{1.1}) has been established by It\^{o} \cite{Ito}.  Recently, (\ref{1.1}) has attracted significant research interest when $b$ is irregular in spatial and/or temporal variables. Based upon the regularising effects of the noise, the strong well-poseness of (\ref{1.1}) has been established for diverse singular drifts. For example, \cite{MNP,Ver,Zvo} investigate bounded measurable drift; \cite{BFGM,Kry21,KR,Nam,RZ,WLWu,XX} focus on integrable drift; and \cite{FGP,GG,TDW,WDGL,WHY} address
time-integrable and spatial H\"{o}lder continuous drift. We further refer to \cite{ZY,Zha05,Zha11} for more details on the case of non-constant diffusion coefficients.

Once the strong well-posedness is established, it becomes natural to explore the fundamental discretizations of the equation. When $b$ is smooth with respect to both spatial and temporal variables, we can approximate (\ref{1.1}) by the classical Euler--Maruyama method
\begin{eqnarray}\label{1.2}
d\tilde{X}_t^n=b(\kappa_n(t),\tilde{X}_{\kappa_n(t)}^n)dt+dB_t, \ \ t\in (0,T], \quad
\tilde{X}_t^n|_{t=0}=x\in{\mathbb R}^d,
\end{eqnarray}
or the Euler--Maruyama ``polygonal'' method
\begin{eqnarray}\label{1.3}
dX_t^n=b(t,X_{\kappa_n(t)}^n)dt+dB_t, \ \ t\in (0,T], \quad
X_t^n|_{t=0}=x\in{\mathbb R}^d,
\end{eqnarray}
where $n\in {\mathbb N}$ and $\kappa_n(t)=\lfloor nt/T \rfloor T/n$ with $\lfloor nt/T \rfloor$ the integer part of $nt/T$. While the strong $L^p$-convergence properties of the above two methods are well understood for smooth $b$ with every $p\in [1,\infty)$, see \cite{PT,TT}, the case of non-smooth coefficients is still an active field of research.

We begin with a brief overview of time-independent drifts. Equations (\ref{1.2}) and (\ref{1.3}) coincide in this setting. If $b$ is piecewise Lipschitz continuous, the strong $L^p$-convergence of $X_t-X_t^n$ was first proved by Leobacher and Sz\"{o}lgyenyi \cite{LS}, M\"{u}ller-Gronbach and Yaroslavtseva
\cite{MY} with a convergence rate of order $1/2$. This result was later strengthened by  Dareiotis and Gerencs\'{e}r for Dini-continuous $b$ ($d=1$ for merely bounded drift). Recently, Neuenkirch and Sz\"{o}lgyenyi \cite{NS} showed a convergence  of order $\min\{3/4,(1+\alpha)/2\}$ in the $L^2$-sense under the conditions $d=1$ and $b\in W^\alpha_2({\mathbb R})$, and in \cite{BDG1} Butkovsky, Dareiotis and Gerencs\'{e}r proved the convergence of order $(1+\alpha)/2$ for general $d\geq 1$. More recently, Dareiotis, Gerencs\'{e}r and L\^{e} \cite{DGL} also recovered the convergence rate $(1+\alpha)/2$ for $b\in L^\infty({\mathbb R}^d)\cap \dot{W}_k^\alpha({\mathbb R}^d)$ with $\alpha\in (0,1)$ and $k\geq d\vee2$. For a comprehensive treatment of non-constant diffusion coefficients, we refer to the seminal works of \cite{BHZ,GK,NT2}.

Literature on the strong convergence of the Euler--Maruyama methods for SDEs with drifts that are time-dependent and temporally continuous is vast as well. When $b(t,x)$ is Lipschitz continuous in $x$ and $1/2$-H\"{o}lder continuous in $t$, the strong $L^p$-convergence rate of $X_t$ to $\tilde{X}_t^n$ is shown to be of order $1/2$, see \cite{KP} for $d\geq 1$ or Yan \cite{Yan} for $d=1$. This result was then generalized by Ngo and Taguchi \cite{NT1} to the drift satisfying one-sided Lipschitz continuity in space and $\eta$-H\"{o}lder continuity in time with $\eta\geq 1/2$. Recently, Pamen and Taguchi \cite{Pamen} established a convergence rate of order $\alpha/2$  for the drift exhibiting $\alpha$-H\"{o}lder continuity in $x$ and $1/2$-H\"{o}lder continuity in $t$. Bao, Huang and Yuan \cite{BHY} further investigated the Dini-continuous $b$ and proved an $L^2$-convergence rate. For comprehensive analyses of convergence between $X_t$ and $X_t^n$, we refer to \cite{Gyo,GR}.

When $b$ is discontinuous (but not overly singular) in spatial and temporal variables, e.g. $b\in L^\infty([0,T]\times {\mathbb R}^d)$, the definition $b(t,X_{\kappa_n(t)}^n)$ in (\ref{1.3}) remains well-defined (similar to \cite{DGL}). However, the representation $b(\kappa_n(t),\tilde{X}_{\kappa_n(t)}^n)$ in (\ref{1.2}) may become ill-defined. This is because the simulation for the standard Euler--Maruyama method may enter a neighborhood of a singularity of $b$, making the method unstable and uncontrollable. To overcome this difficult, one can discuss the convergence in the weak sense or tame the drift $b$ by a suitable approximation (called the tamed Euler--Maruyama method, as borrowed from \cite{HJK}). In the total variation distance between the law of the diffusion and its Euler--Maruyama method: Suo, Yuan and Zhang \cite{SYZ} proved the convergence of order $\alpha/2$ for time-independent drift satisfying an integral condition against some Gaussian measure with $\alpha$-H\"{o}lder type regularity; Bencheikh and Jourdain \cite{BJ} derived a convergence rate of order $1/2$ for the time-dependent bounded drift. For the drift in the Krylov--R\"{o}ckner class
\begin{eqnarray}\label{1.4}
b\in L^q([0,T];L^\rho({\mathbb R}^d;{\mathbb R}^d)), \ \ q,\rho\in (2,\infty) \ \ {\rm and} \ \
2/q+d/\rho<1.
\end{eqnarray}
Jourdain and Menozzi \cite{JS} showed that the marginal density of a tamed Euler--Maruyama method with truncated drift converges at the rate $1/2-1/q-d/(2\rho)$. Recently, using stochastic sewing techniques, L\^{e} and Ling \cite{LL} improved the $L^p$-convergence rate for the tamed Euler--Maruyama method and proved the error is bounded by $n^{-\frac{1}{2}}\log(n)$. However, there are few research works to deal with the numerical analysis of (\ref{1.1}) for $q\leq 2$.

When $q=2$ and $\rho=\infty$, Krylov \cite{Kry25} has proved strong uniqueness for (\ref{1.1}), but the problem of strong existence remains unsolved, rendering the error analysis between (\ref{1.1}) and (\ref{1.3}) equally open. For $q<2$, H\"{o}lder continuity in space must be added to compensate for the loss of time regularity. To get an analogue of (\ref{1.4}), Galeati and Gerencs\'{e}r \cite{GG} supposed that $b\in L^q([0,T];{\mathcal C}^\alpha_b({\mathbb R}^d;{\mathbb R}^d)), \alpha\in (0,1), q\in (2/(1+\alpha),2]$ and by developing a new stochastic sewing lemma, the authors proved the unique strong solvability. Independently, Wei, Hu and Yuan \cite{WHY} showed the strong well-posedness for $b\in L^q([0,T];{\mathcal C}^\alpha({\mathbb R}^d;{\mathbb R}^d))$ ($b$ is H\"{o}lder continuous but not bounded) through the It\^{o}--Tanaka trick. In this paper, we first combine stochastic sewing techniques with the It\^{o}--Tanaka approach to derive the error estimates between $X_t$ and $X_t^n$ for a general low-regularity drift. Second, different from \cite{GG,WHY}, we  provide an explicit construction for the unique strong solution via Picard's successive approximations,  which highlights the utility of stochastic sewing. While this fact is known for Lipschitz or essentially Lipschitz drifts, to the best of our knowledge, this result is new for low-regularity drift. Before giving our main results, we introduce necessary notation.

\subsection{Setup and notation}\label{sec1.1}
We introduce a list of the main notation and conventions adopted throughout the paper.

$\bullet$ ${\mathbb N}$ is the set of natural numbers and ${\mathbb N}_0={\mathbb N}\cup \{0\}$. $\mathbb{R}_+$ represents the set of all positive real numbers. $\nabla$ denotes the gradient of a function with respect to spatial variables.

$\bullet$ For two given constants or functions $C_1$ and $C_2$, $C_1\wedge C_2:=\min\{C_1,C_2\}$ and $C_1\vee C_2:=\max\{C_1,C_2\}$.

$\bullet$ For $\alpha\in (0,1)$, ${\mathcal C}^\alpha_b({\mathbb R}^d)$ is the set consisting of all bounded H\"{o}lder continuous functions with H\"{o}lder exponent $\alpha$, and when $h\in{\mathcal C}^\alpha_b({\mathbb R}^d)$, its norm is given by
\begin{eqnarray*}
\|h\|_{{\mathcal C}^\alpha_b({\mathbb R}^d)}=\sup_{x\in {\mathbb R}^d}|h(x)|+\sup_{x\neq y}\frac{|h(x)-h(y)|}{|x-y|^\beta} =:\|h\|_0+[h]_{\alpha}=:\|h\|_\alpha.
\end{eqnarray*}
Furthermore, if $\nabla^j h$ ($j$-th order gradient with respect to spatial variables), $j=1,2,\ldots,m\in {\mathbb N}$, are bounded and continuous, and $[\nabla^m h]_\alpha$ is finite, then we say $h\in {\mathcal C}_b^{m+\alpha}({\mathbb R}^d)$.  For $h\in {\mathcal C}_b^{m+\alpha}({\mathbb R}^d)$, define
\begin{eqnarray*}
\|h\|_{{\mathcal C}_b^{m+\alpha}({\mathbb R}^d)}
=\sum_{j=0}^m\|\nabla^jh\|_0+[\nabla^mh]_\alpha=:\|h\|_{m+\alpha}.
\end{eqnarray*}
The Lebesgue--H\"{o}lder spaces $L^q([0,T];{\mathcal C}_b^{m+\alpha}({\mathbb R}^d))$ can be defined in a similar way for $q\in[1,\infty]$ and $m\in {\mathbb N}_0$. When $f\in L^q([0,T];{\mathcal C}^{m+\alpha}_b({\mathbb R}^d))$, set
\begin{eqnarray}\label{1.5}
\|f\|_{q,m+\alpha}:=\|f\|_{L^q([0,T];{\mathcal C}_b^{m+\alpha}({\mathbb R}^d))}
 	=\Big[\int_0^T\|f(t,\cdot)\|_{m+\alpha}^qdt\Big]^{\frac{1}{q}},
\end{eqnarray}
where the integral in (\ref{1.5}) is interpreted as the essential supremum when $q=\infty$.
These norms for $q\in [1,\infty]$ are also equivalent to
\begin{equation*}
\left\{\begin{array}{ll} \Big[\sum\limits_{i=0}^m\|\nabla^{i}f\|_{q,0}^q +[\nabla^mf]_{q,\alpha}^q\Big]^{\frac{1}{q}}, \ \ &  \mbox{if} \ \ q<\infty, \\ [0.1cm]\sum\limits_{i=0}^m\|\nabla^{i}f\|_{\infty,0}+[\nabla^mf]_{\infty,\alpha}, \ \ & \mbox{if} \ \ q=\infty,
\end{array}
\right.
\end{equation*}
where
\begin{equation*}
\left\{\begin{array}{ll}
\|\nabla^{i}f\|_{q,0}^q=\int_0^T\|\nabla^{i}f(t,\cdot)\|_0^qdt \ \ {\rm and} \ \ [\nabla^mf]_{q,\beta}^q=\int_0^T[\nabla^mf(t,\cdot)]_\alpha^qdt, \ \ \mbox{if} \ \ q<\infty,  \\ [0.2cm]\|\nabla^if\|_{\infty,0}=\mathop{\rm{esssup}}\limits_{t\in [0,T]}\|\nabla^{i}f(t,\cdot)\|_0 \ \ {\rm and} \ \ [\nabla^mf]_{\infty,\beta}=\mathop{\rm{esssup}}\limits_{t\in [0,T]}[\nabla^mf(t,\cdot)]_\alpha.
\end{array}\right.
\end{equation*}
For simplicity we write $\|f(t,\cdot)\|_{m+\alpha}$ as $\|f(t)\|_{m+\alpha}$ (and similarly for other terms) throughout this paper. Moreover, for a function $f=(f_1,f_2,\ldots,f_d): [0,T]\times \mathbb{R}^d\rightarrow \mathbb{R}^d$ with each component satisfying $f_k\in L^q([0,T];{\mathcal C}_b^{m+\alpha}({\mathbb R}^d))\, (1\leq k\leq d)$, we say $f\in L^q([0,T];{\mathcal C}_b^{m+\alpha}({\mathbb R}^d;{\mathbb R}^d))$ and  denote its norm by
$$
\|f\|_{q,m+\alpha}:=\Big[\sum_{k=1}^d\|f_k\|^2_{q,m+\alpha}\Big]^{\frac{1}{2}}.
$$

$\bullet$ Let $\{{\mathcal P}_t\}_{t\in [0,T]}$ be the Markov transition semigroup associated to the standard Brownian motion. For $h\in {\mathcal C}^\alpha_b({\mathbb R}^d)$ with $\alpha\in (0,1)$, then
\begin{eqnarray}\label{1.6}
{\mathcal P}_{t}h(x)=\int_{{\mathbb R}^d}K(t,x-y)h(y)dy,
\end{eqnarray}
where the Gaussian kernel $K(t,x)$ is given by $(2\pi t)^{-\frac{d}{2}}e^{-\frac{|x|^2}{2t}}$. The following estimates hold
\begin{eqnarray}\label{1.7}
\left\{\begin{array}{ll}
\|\nabla^i{\mathcal P}_th\|_0\leq C(d,\alpha) [h]_\alpha t^{\frac{\alpha-i}{2}}, \ \ i=1,2,
 \\ [0.2cm] \|\partial_t\nabla^i{\mathcal P}_th\|_0\leq C(d,\alpha)[h]_\alpha t^{\frac{\alpha-2-i}{2}}, \ \ i=0,1.
\end{array}\right.
\end{eqnarray}

$\bullet$ Whenever considering a filtered probability space $(\Omega,{\mathcal F},\{{\mathcal F}_t\}_{t\in[0,T]},{\mathbb P})$, we will implicitly
assume that the filtration $\{{\mathcal F}_t\}_{t\in[0,T]}$ satisfies the standard assumptions; in
particular, ${\mathcal F}_0$ is complete. To denote conditional expectations, we use the shortcut
notation ${\mathbb E}^tY :={\mathbb E}[Y|{\mathcal F}_t]$. For $p\geq 1$ we introduce the quantity
\begin{eqnarray*}
\|Y\|_{L^p(\Omega)|{\mathcal F}_t}=\big({\mathbb E}[|Y|^p|{\mathcal F}_t]\big)^{\frac{1}{p}},
\end{eqnarray*}
which is ${\mathcal F}_t$-measurable. Moreover,
\begin{eqnarray}\label{1.8}
\|Y\|_{L^p(\Omega)}=\big(\mathbb{E}|Y|^p\big)^{\frac{1}{p}} =\|\|Y\|_{L^p(\Omega)|{\mathcal F}_t}\|_{L^p(\Omega)}.
\end{eqnarray}
Furthermore, for $Y_1,Y_2\in L^p(\Omega)$ such that $Y_2$ is ${\mathcal F}_t$-measurable, by conditional Jensen's inequality, we have
\begin{eqnarray}\label{1.9}
\|Y_1-{\mathbb E}^tY_1\|_{L^p(\Omega)|{\mathcal F}_t}\leq \|Y_1-Y_2\|_{L^p(\Omega)|{\mathcal F}_t}+\|{\mathbb E}^t(Y_1-Y_2)\|_{L^p(\Omega)|{\mathcal F}_t} \leq 2\|Y_1-Y_2\|_{L^p(\Omega)|{\mathcal F}_t}.
\end{eqnarray}

$\bullet$ Let $g: [0,T]\times \Omega\rightarrow {\mathbb R}^d$ be a measurable bounded function adapted to the filtration $\{{\mathcal F}_t\}_{t\in [0,T]}$. Let $p\geq 2$ and $0\leq s\leq t\leq T$. We define the following norms
\begin{eqnarray}\label{1.10}
\left\{\begin{array}{ll} \|g\|_{{\mathcal C}([s,t];L^p(\Omega))}=\sup\limits_{s\leq r\leq t}\|g(r)\|_{L^p(\Omega)}=:\|g\|_{[s,t],p},  \\ [0.2cm] \|g\|_{L^p(\Omega;{\mathcal C}([s,t]))}=\|\sup\limits_{s\leq r\leq t}|g(r)|\|_{L^p(\Omega)}=:\|g\|_{p,[s,t]}.
\end{array}
\right.
\end{eqnarray}
Moreover, for every $\beta\in (0,1)$, we set
\begin{eqnarray*}
[g]_{p,\beta,[s,t]}:=\Big\|\sup_{s\leq r_1<r_2\leq t}\frac{|g(r_2)-g(r_1)|}{(r_2-r_1)^\beta}\Big\|_{L^p(\Omega)}
\end{eqnarray*}
and
$$
\|g\|_{p,\beta,[s,t]}=\|g\|_{p,[s,t]}+[g]_{p,\beta,[s,t]}.
$$
When $s=0$, we adopt the simplified notations: $\|g\|_{[0,t],p}$ as $\|g\|_{t,p}$, $\|g\|_{p,[0,t]}$ as $\|g\|_{p,t}$, and $\|g\|_{p,\beta,[0,t]}$ as $\|g\|_{p,\beta,t}$.

$\bullet$ We always work on the time interval $[0,T]$.  For $0\leq S_0\leq T_0\leq T$, denote the simplices
\begin{eqnarray*}
\left\{\begin{array}{ll} [S_0,T_0]_\leq^2=\{(s,t), \ S_0\leq s\leq t \leq T_0\},  \\ [0.2cm] [S_0,T_0]_\leq^3=\{(s,u,t), \ S_0\leq s\leq u\leq t\leq T_0\}.
\end{array}
\right.
\end{eqnarray*}
For $(s,t)\in [S_0,T_0]_\leq^2$, define $s_-=s-(t-s)$ and set a modified simplex
\begin{eqnarray*}
\overline{[S_0,T_0]}_\leq^2=\{(s,t)\in[S_0,T_0]_\leq^2, \ s_-\geq S_0\}.
\end{eqnarray*}
Given a map $A: [S_0,T_0]_\leq^2\rightarrow {\mathbb R}^d$, we define $\delta A: [S_0,T_0]_\leq^3\rightarrow {\mathbb R}^d$  by
$$
\delta A_{s,u,t}= A_{s,t}- A_{s,u}- A_{u,t}.
$$

$\bullet$  A measurable function $w: [0,T]_\leq^2 \rightarrow {\mathbb R}_+$ is a control if it is superadditive, i.e.
$$
w(s,u)+w(u,t)\leq w(s,t), \quad \forall \ (s,u,t)\in [0,T]_\leq^3.
$$
For any two controls $w_1, w_2$ and $\theta_1,\theta_2\in {\mathbb R}_+$ such that $\theta_1+\theta_2\geq 1$, then $w=w_1^{\theta_1}w_2^{\theta_2}$ is also a control (see \cite[Erercises 1.8 and 1.9]{FV}). For $b\in L^q([0,T];{\mathcal C}^\alpha_b({\mathbb R}^d;{\mathbb R}^d))$ with $\alpha\in (0,1)$ and $q\in [1,\infty)$, we set
\begin{eqnarray}\label{1.11}
w_{b,\alpha,q}(s,t)=\int_s^t\|b(r)\|_\alpha^qdr.
\end{eqnarray}
Then $w_{b,\alpha,q}$ is a control.

$\bullet$  For $\gamma\in (0,1)$ and $q\in [1,\infty)$, the space $\mathcal{C}_{[S_0,T_0]}^{\gamma,w_{b,\alpha,q}^{1/q}}$ is the set consisting of all continuous functions $\psi$ such that
\begin{eqnarray}\label{1.12}
\begin{split}
\|\psi\|_{\mathcal{C}_{[S_0,T_0]}^{\gamma,w_{b,\alpha,q}^{1/q}}} =\sup_{S_0\leq t\leq T_0}|\psi_t|+\sup_{S_0\leq s<t\leq T_0}\frac{|\psi_t-\psi_s|}{(t-s)^{\gamma}w_{b,\alpha,q}(s,t)^{\frac{1}{q}}}<\infty.
\end{split}
\end{eqnarray}
Then $\mathcal{C}_{[S_0,T_0]}^{\gamma,w_{b,\alpha,q}^{1/q}}$ is a Banach space (similar to the proof of H\"{o}lder space). For $p\in (2\vee (1/\gamma),\infty)$, the set $\mathcal{C}_{[S_0,T_0],p}^{\gamma,w_{b,\alpha,q}^{1/q}}$ consists of all $\{{\mathcal F}_t\}_{t\in[S_0,T_0]}$-adapted processes $\phi$ which has a continuous modification (still denoted by itself) such that
\begin{eqnarray}\label{1.13}
\begin{split}
&\|\phi\|_{\gamma,q,w_{b,\alpha,q},[S_0,T_0],p}:=
\|\phi\|_{\mathcal{C}_{[S_0,T_0],p}^{\gamma,w_{b,\alpha,q}^{1/q}}}=\|\phi\|_{[S_0,T_0],p}+
[\phi]_{\mathcal{C}_{[S_0,T_0],p}^{\gamma,w_{b,\alpha,q}^{1/q}}}
\\=&\sup_{S_0\leq t\leq T_0}\|\phi_t\|_{L^p(\Omega)}+\sup_{S_0\leq s<t\leq T_0}\frac{\|\phi_t-\phi_s\|_{L^p(\Omega)}}{(t-s)^{\gamma}w_{b,\alpha,q}(s,t)^{\frac{1}{q}}}<\infty.
\end{split}
\end{eqnarray}
This space is also a Banach space.  Similarly, we define the Banach space $\mathcal{C}_{p,[S_0,T_0]}^{\gamma,w_{b,\alpha,q}^{1/q}}$ consisting of all $\{{\mathcal F}_t\}_{t\in[S_0,T_0]}$-adapted processes $\varphi$ which has a continuous modification (still denoted by itself) such that
\begin{eqnarray}\label{1.14}
\begin{split}
&\|\varphi\|_{p,\gamma,q,w_{b,\alpha,q},[S_0,T_0]}:=
\|\varphi\|_{\mathcal{C}_{p,[S_0,T_0]}^{\gamma,w_{b,\alpha,q}^{1/q}}}=\|\varphi\|_{p,[S_0,T_0]}+ [\varphi]_{\mathcal{C}_{p,[S_0,T_0]}^{\gamma,w_{b,\alpha,q}^{1/q}}}
\\=&\big\|\sup_{S_0\leq t\leq T_0}|\varphi_t|\big\|_{L^p(\Omega)}+
\Big\|\sup_{S_0\leq s<t\leq T_0}\frac{|\varphi_t-\varphi_s|}{(t-s)^{\gamma}w_{b,\alpha,q}(s,t)^{\frac{1}{q}}}\Big\|_{L^p(\Omega)}<\infty.
\end{split}
\end{eqnarray}
When $S_0=0$, the norms $\|\phi\|_{\gamma,q,w_{b,\alpha,q},[S_0,T_0],p}$ and $\|\varphi\|_{p,\gamma,q,w_{b,\alpha,q},[S_0,T_0]}$ are abbreviated as $\|\phi\|_{\gamma,q,w_{b,\alpha,q},T_0,p}$ and $\|\varphi\|_{p,\gamma,q,w_{b,\alpha,q},T_0}$, respectively.

$\bullet$ The letter $C$  denotes a positive constant, whose values may change in different places. For a parameter or a function $\tilde{\zeta}$, $C(\tilde{\zeta})$ means the constant is only dependent on $\tilde{\zeta}$, and we also write it as $C$ if there is no confusion.

\subsection{Main results}\label{sec1.2} In this subsection, we always assume that
\begin{eqnarray}\label{1.15}
b\in L^q([0,T];{\mathcal C}^\alpha_b({\mathbb R}^d;{\mathbb R}^d)), \ \ \alpha\in (0,1) \ \ {\rm and} \ \ q\in (2/(1+\alpha),\infty].
\end{eqnarray}
The first main result is concerned with the error estimate between $X$ and $X^n$ for $q\geq 2$.
\begin{theorem}\label{the1.1} Let $X_t$ and $X_t^n$ be the unique strong solutions of (\ref{1.1}) and (\ref{1.3}), respectively. If $q\in [2,\infty]$,  then for every $\epsilon\in (0,\alpha/2)$, $\beta\in (0,1/2-1/q+\epsilon)$ and $p\in [2,\infty)$,
\begin{eqnarray}\label{1.16}
\|X-X^n\|_{p,\beta,T}\leq C(d,\alpha,T,p,q,\epsilon,\beta,\|b\|_{q,\alpha})
n^{-\frac{1+\alpha}{2}+\epsilon}.
\end{eqnarray}
\end{theorem}

\begin{remark} \label{rem1.2}
The above $L^p$-convergence rate of order $(1+\alpha)/2$ was first established by Butkovsky, Dareiotis and Gerencs\'{e}r \cite{BDG1} for time-independent and space $\alpha$-H\"{o}lder continuous drift. This order was also obtained by Dareiotis, Gerencs\'{e}r and L\^{e} \cite{DGL} for $b\in L^\infty({\mathbb R}^d)\cap \dot{W}_k^\alpha({\mathbb R}^d)$ with $\alpha\in (0,1)$ and $k\geq d\vee2$,  but $p$ is restricted in $(0,k]$. Here, we generalize the result to all $p\geq 2$ under the minimal assumption  $q\geq 2$, thereby extending previous works from $q=\infty$ to $q\in [2,\infty)$.
\end{remark}

For the case of $q<2$, we also have the following error estimate.
\begin{theorem}\label{the1.3} Let $X_t$ and $X_t^n$ be the unique strong solutions of (\ref{1.1}) and (\ref{1.3}), respectively. If $q\in (2/(1+\alpha),2)$, then for every $\gamma\in (0,(1+\alpha)/2-1/q)$ and $p\in [2,\infty)$,
\begin{eqnarray}\label{1.17}
\|X-X^n\|_{p,\gamma,q,w_{b,\alpha,q},T}\leq  C(d,\alpha,T,p,q,\gamma,\|b\|_{q,\alpha})
n^{\frac{1}{q}-1}.
\end{eqnarray}
\end{theorem}
\begin{remark} \label{rem1.4} When $q$ approaches $2$ from below, the $L^p$-convergence rate   is of order   $1/2$,  which is strictly smaller than the optimal order $(1+\alpha)/2$ that holds when $q\geq 2$. New ideas and tools are needed to address the case  $q<2$.
\end{remark}

Now, let us consider the Picard approximations for $X_t$. Set $\hat{X}^{(0)}_t=x+B_t$ and define successively a sequence of approximations $\{\hat{X}^{(n)}_t\}_{n\geq 1}$ by
\begin{eqnarray}\label{1.18}
\hat{X}^{(n)}_t=x+\int_0^tb(r,\hat{X}^{(n-1)}_r)dr+B_t, \quad t\in (0,T], \ \ n\in \mathbb{N}.
\end{eqnarray}

\begin{theorem}\label{the1.5} Let $X_t$ be as in Theorems \ref{the1.1} and \ref{the1.3} and let $\hat{X}^{(n)}_t$ be defined by (\ref{1.18}). Then, for every $\gamma\in (0,(1+\alpha)/2-1/q)$ and $p\in [2,\infty)$, we have
\begin{eqnarray}\label{1.19}
\lim _{n\rightarrow\infty}\|X-\hat{X}^{(n)}\|_{p,\gamma,q,w_{b,\alpha,q},T}=0.
\end{eqnarray}
\end{theorem}

\subsection{Outline of the proof strategies}\label{sec1.3}
We now provide a schematic overview of the proof methodologies for Theorems \ref{the1.1}, \ref{the1.3} and \ref{the1.5}. Firstly, by (\ref{1.1}) and (\ref{1.3}), one has the following standard error decomposition
\begin{eqnarray}\label{1.20}
X_t-X_t^n=\int_0^t[b(r,X_r)-b(r,X_r^n)]dr+\int_0^t[b(r,X_r^n)-b(r,X_{\kappa_n(r)}^n)]dr.
\end{eqnarray}
We estimate the first term  using It\^{o}--Tanaka's trick. The central challenge lies in establishing rigorous a priori bounds for solutions to the following backward non-homogeneous Kolmogorov equation
\begin{eqnarray*}
\partial_tV(t,x)+\frac{1}{2}\Delta V(t,x)+b(t,x)\cdot \nabla V(t,x)=
	\lambda V(t,x)-b(t,x), \ \ (t,x)\in [0,T)\times {\mathbb R}^d,
\end{eqnarray*}
with terminal condition $V(T,x)=0$. Using regularity estimates for $V$ and It\^{o}'s formula, the first term of the right hand side in (\ref{1.20}) can be represented as
\begin{eqnarray*}
{\mathcal X}_{t,n}+\int_0^t [b(r,X_r^n)-b(r,X_{\kappa_n(r)}^n)]\cdot\nabla V(r,X_r^n)dr,
\end{eqnarray*}
where
$$
{\mathcal X}_{t,n}=[V(t,X_t^n)-V(t,X_t)]+\lambda \int_0^t[V(r,X_r)-V(r,X_r^n)]dr
+\int_0^t[\nabla V(r,X_r)-\nabla V(r,X_r^n)]dB_r.
$$
Therefore,
\begin{eqnarray}\label{1.21}
X_t-X_t^n={\mathcal X}_{t,n}+\int_0^t [b(r,X_r^n)-b(r,X_{\kappa_n(r)}^n)]\cdot[I+\nabla V(r,X_r^n)]dr.
\end{eqnarray}
When estimating the $L^p$-norm for the second term in the right hand side of (\ref{1.21}), we first apply Girsanov's theorem to transform it into
\begin{eqnarray*}
&&\int_0^t [b(r,B^x_r)-b(r,B^x_{\kappa_n(r)})]\cdot[I+\nabla V(r,B^x_r)]dr\nonumber\\ &&\times\exp\Big(\int_0^tb(r,B^x_{\kappa_n(r)})\cdot dB^x_r-\frac{1}{2}\int_0^t|b(r,B^x_{\kappa_n(r)})|^2dr
\Big),
\end{eqnarray*}
where $\{B^x_r\}_{r\in [0,T]}$ is a $d$-dimensional standard Brownian motion starting from $x$, and then bound the above Brownian motion driven ``occupation time functional'' by the standard stochastic sewing techniques. On the other hand, by choosing $\lambda$ large enough, the regularity estimates yield $\|\nabla V\|_{\infty,0}<1/2$. Consequently, the $L^p$ norm ($p\in [2,\infty)$) of the first term in the right hand side of (\ref{1.21}) is bounded above directly by
$$
\frac{1}{2}\|X_t-X_t^n\|_{L^p(\Omega)}+\frac{\lambda}{2} \int_0^t\|X_r-X_r^n\|_{L^p(\Omega)}dr +
C\Big[\int_0^t\|\nabla^2V(r)\|_0^2\|X_r-X_r^n\|_{L^p(\Omega)}^2dr\Big]^{\frac{1}{2}}.
$$
The proof is completed by applying a Gr\"{o}nwall-type estimate.

However, the present strategy is only valid for $q\geq 2$ due to the constraints of Girsanov's theorem. For $q\in (2/(1+\alpha),2)$, we refine the second term in (\ref{1.20}) into two parts
\begin{eqnarray}\label{1.22}
\begin{split} &\int_0^t[b(r,X_r^n)-
b(r,B^x_r-B^x_{\kappa_n(r)}+X_{\kappa_n(r)}^n)]dr\\ &+ \int_0^t[b(r,B^x_r-B^x_{\kappa_n(r)}+X_{\kappa_n(r)}^n)-
b(r,X_{\kappa_n(r)}^n)]dr,
\end{split}
\end{eqnarray}
where the above two terms correspond to ``occupation time functionals'' driven by Markov processes. To estimate these terms, we develop a refined stochastic sewing lemma for separate analysis. The second term in (\ref{1.22}) is a difference of average of $b$ along $X_{\kappa_n}$ with a perturbation of $B^x-B^x_{\kappa_n}$. By applying the stochastic sewing to this term and incorporating properties of Brownian motion, we obtain
\begin{eqnarray*}
\begin{split} &\Big\|\int_s^t[b(r,B^x_r-B^x_{\kappa_n(r)}+X_{\kappa_n(r)}^n)-
b(r,X_{\kappa_n(r)}^n)]dr\Big\|_{L^p(\Omega)} \\ \leq& C w_{b,\alpha,q}(s,t)^{\frac{1}{q}}(t-s)^{\varepsilon_1} n^{-1-\frac{\alpha}{2}+\frac{1}{q}+\varepsilon_1}, \quad \forall \ \varepsilon_1\in [0,(1+\alpha)/2-1/q],
\end{split}
\end{eqnarray*}
and the convergence rate of order $1+\alpha/2-1/q$ may be optimal. For the first term, hindered by $|X_r^n-X_{\kappa_n(r)}^n|$, stochastic sewing only provides order $1-1/q$. Consequently, we employ a hybrid method (PDE+stochastic sewing) for $q\geq 2$ (Theorem \ref{the1.1})  and a pure stochastic sewing approach for $q<2$ (Theorem \ref{the1.3}) to calculate errors in Sections 3 and 4, respectively. These methodologies may have broader potential applications in distinct mathematical regimes.

To prove Theorem \ref{the1.5}, we decompose the difference between $\hat{X}^{(n)}$ and $\hat{X}^{(m)}$ by
\begin{eqnarray*}
\hat{X}^{(n)}_t-\hat{X}^{(m)}_t=\sum_{i=1}^{n-m}\int_0^t[b(r,\hat{X}^{(n-i)}_r)-b(r,\hat{X}^{(n-i-1)}_r)]dr, \quad \forall \ n>m\geq 1.
\end{eqnarray*}
We partition the time interval $[0, T]$ into $N_0\in \mathbb{N}$ subintervals with $t_i=iT/N_0$ ($i=0,1,\ldots,N_0$) for some large enough $N_0$. On time interval $[t_0,t_1]$, using the stochastic sewing techniques developed in Section 4, the Kolmogorov--Chentsov type continuity theorem (Lemma \ref{lem2.5}) and H\"{o}lder's inequality, for every  $\gamma\in (0,(1+\alpha)/2-1/q)$ and every $p\in [2,\infty)$ we show that
\begin{eqnarray}\label{1.23}
\|\hat{X}^{(n)}-\hat{X}^{(m)}\|_{p,\gamma,q,w_{b,\alpha,q},t_1}
\leq C(d,\alpha,T,\gamma,p,\|b\|_{q,\alpha})2^{-m}.
\end{eqnarray}
This implies $\{\hat{X}^{(n)}_\cdot\}_{n\geq 1}$ is a Cauchy sequence in the space $\mathcal{C}_{p,[t_0,t_1]}^{\gamma,w_{b,\alpha,q}^{1/q}}$. By the completeness of $\mathcal{C}_{p,[t_0,t_1]}^{\gamma,w_{b,\alpha,q}^{1/q}}$ and the uniqueness of the strong solution of (\ref{1.1}), we conclude the proof on the time interval $[t_0,t_1]$. We then repeat the arguments for $t\in [t_0,t_1]$ to extend the estimate (\ref{1.23}) to $[t_1,t_2]$. Continuing this procedure with finitely many steps, we obtain a uniform estimate over the whole time interval.

\section{Useful lemmas}\label{sec2}\setcounter{equation}{0}
Let $b: [0,T]\times{\mathbb R}^d\rightarrow{\mathbb R}^d$ and $f: [0,T]\times{\mathbb R}^d\rightarrow{\mathbb R}$ be Borel measurable functions. Consider the following Kolmogorov equation
\begin{eqnarray}\label{2.1}
\left\{\begin{array}{ll}
\partial_{t}v(t,x)=\frac{1}{2}\Delta v(t,x)+b(t,x)\cdot \nabla v(t,x)
\\ [0.2cm] \qquad\qquad \ \ -\lambda v(t,x)+f(t,x), \ (t,x)\in (0,T]\times {\mathbb R}^d, \\ [0.2cm]
v(t,x)|_{t=0}=0, \  x\in{\mathbb R}^d,  \end{array}\right.
\end{eqnarray}
where $\lambda>0$ is a real number. If $v\in
L^1([0,T];{\mathcal C}^2({\mathbb R}^d))\cap W^{1,1}([0,T];{\mathcal C}({\mathbb R}^d))$ such that (\ref{2.1}) holds true for almost all $(t,x)\in (0,T)\times {\mathbb R}^d$, then the unknown function $v$ is said to be a strong solution. Let us first give the existence and uniqueness of strong solutions for (\ref{2.1}).

\begin{lemma} \label{lem2.1} \textbf{(Existence and uniqueness)}  Let  $b\in L^q([0,T];
{\mathcal C}^\alpha_b({\mathbb R}^d;{\mathbb R}^d))$ and $f\in L^q([0,T];{\mathcal C}^\alpha_b({\mathbb R}^d))$ with $\alpha\in (0,1)$ and $q\in [2,\infty]$. Then there exists a unique strong solution $v\in L^q([0,T];{\mathcal C}_b^{2+\alpha}({\mathbb R}^d))\cap W^{1,q}([0,T];{\mathcal C}_b^\alpha({\mathbb R}^d))$ to (\ref{2.1}). Moreover, we have
\begin{eqnarray}\label{2.2}
\left\{\begin{array}{ll}\nabla v\in L^\infty([0,T];{\mathcal C}_b^{1+\alpha-\frac{2}{q}}({\mathbb R}^d;{\mathbb R}^d)),  &  {\rm if} \ \ q\neq 2/\alpha,\\ [0.2cm] \nabla v\in \mathop{\cap}\limits_{\theta\in [0, 1)}L^\infty([0,T];{\mathcal C}_b^\theta({\mathbb R}^d;{\mathbb R}^d)) \cap \mathop{\cap}\limits_{\tilde{q}\in [q, \infty)}L^{\tilde{q}}([0,T];{\mathcal C}_b^1({\mathbb R}^d;{\mathbb R}^d)), \quad & {\rm if} \ \ q=2/\alpha,\\ [0.2cm]
\nabla v\in  L^{\frac{2q}{2-q\alpha}}([0,T];{\mathcal C}_b^1({\mathbb R}^d;{\mathbb R}^d)), & {\rm if} \ \ q\in [2,2/\alpha).
\end{array}\right.
\end{eqnarray}
Furthermore, there is a real number $\varepsilon>0$ such that for large enough $\lambda>0$,
\begin{eqnarray}\label{2.3}
\sup_{0\leq t\leq T}\|\nabla v(t)\|_0\leq C(d,\alpha,T,q,\|b\|_{q,\alpha})\lambda^{-\varepsilon}
[f]_{q,\alpha}.
\end{eqnarray}
\end{lemma}
\begin{proof}  For $q\in [2,\infty]$, we have $L^q([0,T];
{\mathcal C}^\alpha_b({\mathbb R}^d))\subset L^2([0,T];
{\mathcal C}^\alpha_b({\mathbb R}^d))$. In view of \cite[Corollary 2.1]{TDW}, (\ref{2.1}) exists a unique strong solution which satisfies (\ref{2.3}). When $q>2/\alpha$, it is straightforward from  \cite[Lemma 2.1]{WDGL} to get the conclusions. For $q\in [2,2/\alpha]$,
if $b\equiv 0$, with the help of \cite[Theorem 3.1]{Krylov02}, then $v\in L^q([0,T];{\mathcal C}_b^{2+\alpha}({\mathbb R}^d))\cap W^{1,q}([0,T];{\mathcal C}_b^\alpha({\mathbb R}^d))$.  Moreover, by the semigroup representation, then
\begin{eqnarray}\label{2.4}
v(t,x)=\int_0^te^{-\lambda (t-r)}{\mathcal P}_{t-r}f(r,x)dr,
\end{eqnarray}
where $\{{\mathcal P}_t\}_{t\in [0,T]}$ is given by (\ref{1.6}). This, combined with the classical heat kernel estimates (or see the proof details of \cite[Lemma 2.2]{TDW}), leads to
\begin{eqnarray*}
\left\{\begin{array}{ll}
\nabla v\in L^\infty([0,T];{\mathcal C}_b^{1+\alpha-\frac{2}{q}}({\mathbb R}^d;{\mathbb R}^d)), & {\rm if} \  \ q\in [2,2/\alpha),\\ [0.2cm] \nabla v\in \mathop{\cap}\limits_{\theta\in [0, 1)}L^\infty([0,T];{\mathcal C}_b^\theta({\mathbb R}^d;{\mathbb R}^d)), \quad &  {\rm if} \  \ q=2/\alpha.
\end{array}\right.
\end{eqnarray*}
For the second derivatives, by (\ref{2.4}) and (\ref{1.7}) we have
\begin{eqnarray*}
\|\nabla^2 v(t)\|_0\leq C(d,\alpha)\int_0^te^{-\lambda (t-r)}[f(r)]_\alpha (t-r)^{\frac{\alpha}{2}-1}dr.
\end{eqnarray*}
Coupled with the Hardy--Littlewood--Sobolev convolution inequality (see \cite[Theorem 4.5.3]{Hormander}), this gives rise to
\begin{eqnarray*}
\left\{\begin{array}{ll}
\|\nabla^2 v\|_{\frac{2q}{2-q\alpha},0}\leq C(d,\alpha,q) [f]_{q,\alpha}, & {\rm if} \  \ q\in [2,2/\alpha),\\ [0.2cm] \|\nabla^2 v\|_{\tilde{q},0}\leq C(d,\alpha,q,\tilde{q}) [f]_{q,\alpha}, \ \  \forall \ \tilde{q}\in [q,\infty), \quad &  {\rm if} \  \ q=2/\alpha.
\end{array}\right.
\end{eqnarray*}
This, together with $\nabla^2 v \in L^q([0,T];{\mathcal C}_b^{\alpha}({\mathbb R}^d;\mathbb{R}^{d\times d}))$, implies
\begin{eqnarray*}
\left\{\begin{array}{ll}
\nabla^2 v\in L^{\frac{2q}{2-q\alpha}}([0,T];{\mathcal C}_b({\mathbb R}^d;\mathbb{R}^{d\times d})), & {\rm if} \ \ q\in [2,2/\alpha),\\ [0.2cm] \nabla^2 v\in \mathop{\cap}\limits_{\tilde{q}\in [q, \infty)}L^{\tilde{q}}([0,T];{\mathcal C}_b({\mathbb R}^d;\mathbb{R}^{d\times d})), \quad & {\rm if} \  \ q=2/\alpha.
\end{array}\right.
\end{eqnarray*}
Thus, (\ref{2.2}) holds true. For general $b$, we define a mapping ${\mathcal T}$ on $L^\infty([0,T];{\mathcal C}_b^{1+\alpha}({\mathbb R}^d))$ by
\begin{eqnarray}\label{2.5}
{\mathcal T} \tilde{v}(t,x)=\int_0^te^{-\lambda (t-r)}{\mathcal P}_{t-r}[b(r,\cdot)\cdot \nabla \tilde{v}(r,\cdot)+f(r,\cdot)](x)dr.
\end{eqnarray}
Using the method of continuity, we conclude our claims. \end{proof}

\begin{remark}\label{rem2.2} Notice that for a second order parabolic partial differential equation we can ``trade'' space-regularity against time-regularity at a cost of one time derivative
for two space derivatives, and from (\ref{2.2}) we also get the H\"{o}lder continuity of $\nabla v$ in the temporal variable. For simplicity, we take $b\equiv 0$ to prove this continuity (for more details to see \cite[Lemma 2.2]{WHY}). By (\ref{2.4}), for $0\leq s<t\leq T$,
\begin{eqnarray*}
\begin{split}
&\nabla v(t,x)-\nabla v(s,x)=\int_s^te^{-\lambda (t-r)}\nabla {\mathcal P}_{t-r}f(r,x) dr\\ &+\int_0^s[e^{-\lambda (t-r)}-e^{-\lambda (s-r)}]\nabla{\mathcal P}_{t-r}f(r,x)dr
+\int_0^se^{-\lambda (s-r)}dr \int_{s-r}^{t-r}\partial_\tau\nabla{\mathcal P}_\tau f(r,x)d\tau.
\end{split}
\end{eqnarray*}
This, along with (\ref{1.7}), results in
\begin{eqnarray*}
\begin{split}
|\nabla v(t,x)-\nabla v(s,x)|\leq &C\Big[\int_s^t[f(r)]_\alpha (t-r)^{\frac{\alpha-1}{2}}dr+\int_0^s [f(r)]_\alpha (t-r)^{\frac{\alpha-1}{2}} dr (t-s)\Big]\nonumber\\
&+C\int_0^s[f(r)]_\alpha dr\int_{s-r}^{t-r}\tau^{\frac{1}{q}-1+\epsilon}\tau^{\frac{\alpha-1}{2}-
\frac{1}{q}-\epsilon}d\tau
\\ \leq&  C(d,\alpha,\epsilon,T)[f]_{q,\alpha}(t-s)^{\frac{1+\alpha}{2}-\frac{1}{q}-\epsilon}, \ \ \forall \ \epsilon\in (0,(1+\alpha)/2-1/q).
 \end{split}
 \end{eqnarray*}
Thus, for $f\in L^q([0,T];{\mathcal C}^\alpha_b({\mathbb R}^d))$ with $\alpha\in (0,1)$ and $q\in [2,\infty]$, we have
\begin{eqnarray}\label{2.6}
 \nabla v\in \mathop{\cap}\limits_{\theta\in [0, 1+\alpha-2/q)}{\mathcal C}^{\frac{\theta}{2}}_b([0,T];{\mathcal C}_b({\mathbb R}^d;{\mathbb R}^d)).
\end{eqnarray}
\end{remark}

To prove Theorems \ref{the1.1}, \ref{the1.3} and \ref{the1.5}, we also need a stochastic sewing lemma. The original versions can be found in \cite[Theorems 2.1 and 2.3]{Le20} (or \cite[Theorem 3.1]{Le23}). However, this version can not be  used to prove Theorems \ref{the1.1}, \ref{the1.3} and \ref{the1.5} directly, so we introduce a modified version which is inspired by \cite[Lemma 2.5]{GG} and \cite[Lemma 3.1]{BDG2}. Since the proof is similar to that of \cite[Lemma 2.5]{GG}, we omit the proof details.
\begin{lemma} \label{lem2.3} Let $p\in [2,\infty)$ and $(S_0,T_0)\in [0,T]_\leq^2$. Let $\{A_{s,t}\}_{(s,t)\in [S_0,T_0]_\leq^2}$ be a two-parameter stochastic process with values in ${\mathbb R}^d$ which is $L^p$-integrable and for every $(s,t)\in [S_0,T_0]_\leq^2$, $A_{s,t}$ is ${\mathcal F}_t$-measurable.
Assume that there are constants $\Gamma_1,\Gamma_2\geq 0$, $\varepsilon_1,\varepsilon_2>0$ and two controls $w_1,w_2$ such that
\begin{eqnarray}\label{2.7}
\|A_{s,t}\|_{L^p(\Omega)}\leq \Gamma_1 w_1(s_-,t)^{\frac{1}{2}}(t-s)^{\varepsilon_1}, \quad \forall \ (s,t)\in \overline{[S_0,T_0]}_\leq^2
\end{eqnarray}
and
\begin{eqnarray}\label{2.8}
\|{\mathbb E}^{s_-}\delta A_{s,u,t}\|_{L^p(\Omega)}\leq \Gamma_2 w_2(s_-,t)(t-s)^{\varepsilon_2}, \ \forall \ (s,t)\in \overline{[S_0,T_0]}_\leq^2, \ u=(s+t)/2.
\end{eqnarray}
Furthermore, suppose that there exists a stochastic process ${\mathcal A}=\{{\mathcal A}_t\}_{t\in[S_0,T_0]}$ satisfying  for any $(s,t)\in [S_0,T_0]_\leq^2$
\begin{eqnarray}\label{2.9}
{\mathcal A}_t-{\mathcal A}_s=\lim_{k\rightarrow \infty}\sum_{i=1}^{k-1}A_{s+i\frac{t-s}{k},s+(i+1)\frac{t-s}{k}}, \quad {\rm in} \ \ L^1(\Omega).
\end{eqnarray}
Then there exist constants $K_1,K_2>0$ which depend only on $\varepsilon_1,\varepsilon_2,p$ and $d$, it holds that for any $(s,t)\in [S_0,T_0]_\leq^2$
\begin{eqnarray}\label{2.10}
\|{\mathcal A}_t-{\mathcal A}_s\|_{L^p(\Omega)}\leq K_1\Gamma_1 w_1(s,t)^{\frac{1}{2}}(t-s)^{\varepsilon_1}+ K_2\Gamma_2 w_2(s,t)(t-s)^{\varepsilon_2}.
\end{eqnarray}
\end{lemma}

To prove Theorems \ref{the1.3} and \ref{the1.5}, we additionally require two key estimates.
\begin{lemma} \label{lem2.4} Assume $b\in L^q([0,T];{\mathcal C}^\alpha_b({\mathbb R}^d;{\mathbb R}^d))$ with $\alpha\in (0,1)$ and $q\in (2/(1+\alpha),2)$. Let $X_t$, $X_t^n$ and $\hat{X}_t^{(n)}$ are given by (\ref{1.1}), (\ref{1.3}) and (\ref{1.18}), respectively. We denote $\varphi_t=X_t-B_t$, $\varphi^n_t=X^n_t-B_t$ and $\hat{\varphi}^{(n)}_t=\hat{X}^{(n)}_t-B_t$. Then for every $(s,t)\in [0,T]_\leq^2$ and $g_t\in \{\varphi_t, \varphi^n_t,\hat{\varphi}^{(n)}_t\}$, we have
\begin{eqnarray}\label{2.11}
{\mathbb E}^s|g_t-{\mathbb E}^s
g_t|\leq C(d,\alpha,T,\|b\|_{q,\alpha})w_{b,\alpha,q}(s,t)^{\frac{1}{q}}(t-s)^{1+\frac{\alpha}{2}-\frac{1}{q}},
\end{eqnarray}
where $w_{b,\alpha,q}(s,t)$ is given by (\ref{1.11}).
\end{lemma}
\begin{proof} The proof for $g_t=\varphi_t$ has been given in \cite[Lemma 2.1]{GG}, it is sufficient to prove $g_t\in \{\varphi^n_t,\hat{\varphi}^{(n)}_t\}$. The proof for $g_t=\hat{\varphi}^{(n)}_t$ is similar to that of $g_t=\varphi^n_t$. We only give the details for $g_t=\varphi^n_t$. Observe that
\begin{eqnarray}\label{2.12}
\varphi_t^n=x+\int_0^tb(r,B_{\kappa_n(r)}+\varphi_{\kappa_n(r)}^n)dr,
\end{eqnarray}
then $\varphi_t^n$ is ${\mathcal F}_s$-measurable if $t\in [s,\kappa_n(s)+T/n]$. In this case, (\ref{2.11}) holds true ad hoc.

If $t>\kappa_n(s)+T/n$, we take
\begin{eqnarray}\label{2.13}
\tilde{\varphi}_t^n=x+\varphi_{\kappa_n(s)+\frac{T}{n}}^n+\int_{\kappa_n(s)+\frac{T}{n}}^tb(r,B_s+{\mathbb E}^s\varphi_{\kappa_n(r)}^n)dr,
\end{eqnarray}
then $\tilde{\varphi}_t^n$ is ${\mathcal F}_s$ measurable.  By (\ref{1.9}), one gets
\begin{eqnarray*}
{\mathbb E}^s|\varphi_t^n-{\mathbb E}^s
\varphi_t^n|\leq 2{\mathbb E}^s|\varphi_t^n-\tilde{\varphi}_t^n|.
\end{eqnarray*}
This, together with (\ref{2.12}) and (\ref{2.13}), leads to
\begin{eqnarray}\label{2.14}
\begin{split}
{\mathbb E}^s|\varphi_t^n-{\mathbb E}^s
\varphi_t^n|&\leq 2{\mathbb E}^s\Big|\int_{\kappa_n(s)+\frac{T}{n}}^t[b(r,B_{\kappa_n(r)}+\varphi_{\kappa_n(r)}^n)-b(r,B_s
+{\mathbb E}^s\varphi_{\kappa_n(r)}^n)]dr\Big|
\\
 &\leq 2\int_{\kappa_n(s)+\frac{T}{n}}^t[b(r)]_\alpha\big[{\mathbb E}^s|B_{\kappa_n(r)}-B_s|^\alpha+
 {\mathbb E}^s|\varphi_{\kappa_n(r)}^n-{\mathbb E}^s\varphi_{\kappa_n(r)}^n|^\alpha\big
 ]dr.
 \end{split}
\end{eqnarray}
For $r\geq \kappa_n(s)+T/n$, we have $\kappa_n(r)>s$. If one chooses
\begin{eqnarray*}
\hat{\varphi}_{\kappa_n(r)}^n=x+\varphi_s^n+\int_s^{\kappa_n(r)}b(\tau,B_s+{\mathbb E}^s\varphi_{\kappa_n(\tau)}^n)d\tau,
\end{eqnarray*}
then $\hat{\varphi}_{\kappa_n(r)}^n$ is ${\mathcal F}_s$ measurable. By (\ref{2.12}) and (\ref{1.9}), we obtain
\begin{eqnarray}\label{2.15}
\begin{split}
{\mathbb E}^s|\varphi_{\kappa_n(r)}^n-{\mathbb E}^s\varphi_{\kappa_n(r)}^n|&\leq 2{\mathbb E}^s |\varphi_{\kappa_n(r)}^n-\hat{\varphi}_{\kappa_n(r)}^n|\\
 &=
 2{\mathbb E}^s\Big|\int_s^{\kappa_n(r)}
[b(\tau,B_{\kappa_n(\tau)}+\varphi_{\kappa_n(\tau)}^n)-b(\tau,B_s+{\mathbb E}^s\varphi_{\kappa_n(\tau)}^n)]d\tau \Big| \\
 &\leq 4\int_s^r
\|b(\tau)\|_0d\tau \leq 4w_{b,\alpha,q}(s,r)^{\frac{1}{q}}(r-s)^{1-\frac{1}{q}}.
\end{split}
\end{eqnarray}
In view of conditional Jensen's inequality and (\ref{2.15}),
\begin{eqnarray*}
{\mathbb E}^s|\varphi_{\kappa_n(r)}^n-{\mathbb E}^s\varphi_{\kappa_n(r)}^n|^\alpha\leq \big[{\mathbb E}^s|\varphi_{\kappa_n(r)}^n-{\mathbb E}^s\varphi_{\kappa_n(r)}^n|\big]^\alpha \leq C(d,\alpha,T,\|b\|_{q,\alpha})(r-s)^{(1-\frac{1}{q})\alpha}.
\end{eqnarray*}
Combined with (\ref{2.14}), this further implies
\begin{eqnarray*}
\begin{split}
{\mathbb E}^s|\varphi_t^n-{\mathbb E}^s
\varphi_t^n|\leq& C(d,\alpha,T,\|b\|_{q,\alpha})\int_{\kappa_n(s)+\frac{T}{n}}^t[b(r)]_\alpha\big[(r-s)^{\frac{\alpha}{2}}+ (r-s)^{(1-\frac{1}{q})\alpha}\big
 ]dr \\ \leq& C(d,\alpha,T,\|b\|_{q,\alpha})w_{b,\alpha,q}(s,t)^{\frac{1}{q}}\big[(t-s)^{1+\frac{\alpha}{2}-\frac{1}{q}}+
(t-s)^{(1-\frac{1}{q})(1+\alpha)}] \\ \leq& C^2
(t-s)^{(1-\frac{1}{q})(1+\alpha)}.
\end{split}
\end{eqnarray*}
Observe that $q>2/(1+\alpha)$, then there exists some $k_0\in {\mathbb N}$ such that
$$
\Big(1-\frac{1}{q}\Big)\sum_{i=1}^{k_0}\alpha^i<\frac{\alpha}{2} \ \ {\rm and} \ \ \Big(1-\frac{1}{q}\Big)\sum_{i=1}^{k_0+1}\alpha^i\geq \frac{\alpha}{2}.
$$
We repeat the proceeding arguments $k_0\in {\mathbb N}$ times again and get
\begin{eqnarray*}
\begin{split}
{\mathbb E}^s|\varphi_t^n-{\mathbb E}^s
\varphi_t^n|\leq& C^{k_0+2}w_{b,\alpha,q}(s,t)^{\frac{1}{q}}\big[(t-s)^{1+\frac{\alpha}{2}-\frac{1}{q}}+
(t-s)^{(1-\frac{1}{q})(1+\alpha+\alpha^2+\ldots+\alpha^{k_0+1})}\big] \\ \leq& C^{k_0+2}w_{b,\alpha,q}(s,t)^{\frac{1}{q}}(t-s)^{1+\frac{\alpha}{2}-\frac{1}{q}},
\end{split}
 \end{eqnarray*}
which implies (\ref{2.11}).  \end{proof}

To prove the convergence  in the weighted space, we cite the Kolmogorov--Chentsov type continuity theorem (see \cite[Theorem 2.1, p25]{RY}, \cite[Theorem 1.3]{WL}  and \cite[Lemma A.3]{GG}) as the following lemma.
\begin{lemma} \label{lem2.5} Let $\Xi_{\cdot}: [0,T]\rightarrow \mathbb{R}^d$ be a stochastic process defined on the probability space $(\Omega,{\mathcal F},{\mathbb P})$. Suppose there exist $\alpha\in (0,1]$, $\beta\in [0,1]$, $m\in (1/\alpha,\infty)$, a control $w:[0,T]^2\rightarrow [0,\infty)$, a constant $K_0>0$ such that
\begin{eqnarray}\label{2.16}
\big(\mathbb{E}|\Xi_{s,t}|^m\big)^{\frac{1}{m}}\leq K_0w(s,t)^\beta(t-s)^\alpha, \quad 0\leq s<t\leq T,
\end{eqnarray}
where $\Xi_{s,t}=\Xi_t-\Xi_s$. Then $\Xi$ has a continuous realization $\tilde{\Xi}$, namely, there exists $\Omega_0$ such that ${\mathbb P}(\Omega_0)=1$ and for each $\omega\in\Omega_0$, $\Xi_t(\omega)=\tilde{\Xi}_t(\omega)$ and $\tilde{\Xi}_t(\omega)$ is a continuous function of $t$. Moreover, for any $0<\gamma<\alpha-1/m$, there exists a positive constant $C=C(\alpha,\gamma,m)$ such that
\begin{eqnarray}\label{2.17}
\Big(\mathbb{E}\sup_{0\leq s<t\leq T}\frac{|\Xi_{s,t}|^m}{(t-s)^{m\gamma}w(s,t)^{m\beta}}\Big)^{\frac{1}{m}}\leq CK_0.
\end{eqnarray}
\end{lemma}

For further use, we use the Newton--Leibniz formula to give two elementary estimates.
\begin{lemma} \label{lem2.6} Let $f$ be a Borel measurable function in $[0,T]\times \mathbb{R}^d$. For $r\in [0,T]$ and $x_1,x_2,x_3,x_4\in {\mathbb R}^d$, we set $f(r,x_1,x_2,x_3,x_4):=f(r,x_1)-f(r,x_2)-f(r,x_3)+f(r,x_4)$.

(i) If $f\in L^1([0,T];{\mathcal C}_b^2({\mathbb R}^d;{\mathbb R}^d))$, then
\begin{eqnarray}\label{2.18}
\begin{split}
|f(r,x_1,x_2,x_3,x_4)|
 \leq&\|\nabla f(r)\|_0|x_1-x_2-x_3+x_4|\\ &+ \frac{1}{2}\|\nabla^2 f(r)\|_0\big[|x_2-x_4|+|x_1-x_3|\big]|x_1-x_2|.
\end{split}
\end{eqnarray}

(ii) If $f\in L^1([0,T];{\mathcal C}_b^{1+\beta}({\mathbb R}^d;{\mathbb R}^d))$ with $\beta\in (0,1)$, then
\begin{eqnarray}\label{2.19}
\begin{split}
|f(r,x_1,x_2,x_3,x_4)|
 \leq&\|\nabla f(r)\|_0|x_1-x_2-x_3+x_4|\\&+ \frac{1}{1+\beta}[\nabla f(r)]_\beta\big[|x_2-x_4|^\beta+|x_1-x_3|^\beta\big]|x_1-x_2|.
\end{split}
\end{eqnarray}
\end{lemma}
\begin{proof} We only give the calculation details for (\ref{2.18}). For every fixed $r\in [0,T]$, using Newton--Leibniz's formula to $f(r,\cdot)$ we achieve
\begin{eqnarray*}
\begin{split}
f(r,x_1,x_2,x_3,x_4)=&\int_0^1[\nabla f(r,x_2+\tau(x_1-x_2))-\nabla f(r,x_4+\tau(x_3-x_4))]\cdot(x_1-x_2)d\tau
  \nonumber\\ &
+\int_0^1\nabla f(r,x_4+\tau(x_3-x_4))\cdot(x_1-x_2-x_3+x_4)d\tau.
\end{split}
\end{eqnarray*}
By the mean value inequality,
\begin{eqnarray*}
\begin{split}
&|\nabla f(r,x_2+\tau(x_1-x_2))-\nabla f(r,x_4+\tau(x_3-x_4))|\nonumber\\
 \leq& \|\nabla^2f(r)\|_0\big[(1-\tau)|x_2-x_4|+\tau|x_1-x_3|\big].
\end{split}
\end{eqnarray*}
Hence, $|f(r,x_1,x_2,x_3,x_4)|$ is bounded above by
\begin{eqnarray*}
\begin{split}
&\|\nabla f(r)\|_0|x_1-x_2-x_3+x_4|\\ &+ \|\nabla^2 f(r)\|_0\int_0^1\big[(1-\tau) |x_2-x_4|+\tau|x_1-x_3|\big]d\tau |x_1-x_2|
\nonumber\\
 \leq&\|\nabla f(r)\|_0|x_1-x_2-x_3+x_4|+\frac{1}{2}\|\nabla^2 f(r)\|_0\big[|x_2-x_4|+|x_1-x_3|\big]|x_1-x_2|.
\end{split}
\end{eqnarray*}
This completes the proof. \end{proof}

\section{Proof of Theorem \ref{the1.1}}\label{sec3}\setcounter{equation}{0}
 Let $X_t$ and $X_t^n$ be given by (\ref{1.1}) and (\ref{1.3}). For every $n\in {\mathbb N}$ and $t\in [0,T]$, then
\begin{eqnarray}\label{3.1}
\begin{split}
X_t-X_t^n=&\int_0^tb(r,X_r)dr- \int_0^tb(r,X_{\kappa_n(r)}^n)dr
\\ =&\int_0^t[b(r,X_r)-b(r,X_r^n)]dr+\int_0^t[b(r,X_r^n)-b(r,X_{\kappa_n(r)}^n)]dr=:I_{t,n}+J_{t,n}.
\end{split}
\end{eqnarray}
We divide the calculations into four steps.

\smallskip
\textbf{Step 1:} To estimate $I_{t,n}$.  Let $\lambda>0$ be a large enough real number, consider the following vector-valued Cauchy problem:
\begin{eqnarray}\label{3.2}
\left\{\begin{array}{ll}\partial_tV(t,x)+\frac{1}{2}\Delta V(t,x)+b(t,x)\cdot\nabla V(t,x)\\  [0.2cm] \qquad\qquad=\lambda V(t,x)-b(t,x), \ \ (t,x)\in [0,T)\times {\mathbb R}^d,\\ [0.2cm]
V(t,x)|_{t=T}=0, \  x\in{\mathbb R}^d.
\end{array}\right.
\end{eqnarray}
By Lemma \ref{lem2.1}, there is a unique strong solution
$$
V\in L^q([0,T];{\mathcal C}_b^{2+\alpha}({\mathbb R}^d;{\mathbb R}^d))\cap W^{1,q}([0,T];{\mathcal C}_b^\alpha({\mathbb R}^d;{\mathbb R}^d))
$$
to (\ref{3.2}). Moreover, by (\ref{2.2}) and (\ref{2.6}),
\begin{eqnarray}\label{3.3}
\left\{\begin{array}{ll}\nabla V\in L^\infty([0,T];{\mathcal C}_b^{1+\alpha-\frac{2}{q}}({\mathbb R}^d;{\mathbb R}^{d\times d})), & {\rm if} \  \ q\neq 2/\alpha,\\ [0.2cm]
\nabla V\in \mathop{\cap}\limits_{\theta\in [0, 1)}L^\infty([0,T];{\mathcal C}_b^\theta({\mathbb R}^d;{\mathbb R}^{d\times d}))\cap
\mathop{\cap}\limits_{\tilde{q}\in [q, \infty)}L^{\tilde{q}}([0,T];{\mathcal C}_b^1({\mathbb R}^d;{\mathbb R}^{d\times d}))
, & {\rm if} \  \ q=2/\alpha, \\ [0.2cm] \nabla V\in L^{\frac{2q}{2-q\alpha}}([0,T];{\mathcal C}_b^1({\mathbb R}^d;{\mathbb R}^{d\times d})), & {\rm if} \ \ q\in [2,2/\alpha),
\\ [0.2cm]
\nabla V\in \mathop{\cap}\limits_{\theta\in [0, 1+\alpha-2/q)}{\mathcal C}^{\frac{\theta}{2}}_b([0,T];{\mathcal C}_b({\mathbb R}^d;{\mathbb R}^{d\times d})), \quad & {\rm if} \  \ q\in [2,\infty].
\end{array}\right.
\end{eqnarray}

Furthermore,  (\ref{2.3}) implies
\begin{eqnarray}\label{3.4}
\sup_{0\leq t\leq T}\|\nabla V(t)\|_0\leq C(d,\alpha,T,q,\|b\|_{q,\alpha})\lambda^{-\varepsilon}
[b]_{q,\alpha}<\frac{1}{2},
\end{eqnarray}
provided that $\lambda$ is large enough. We define $\Phi(t,x)=x+V(t,x)$. Then $\Phi$ forms a diffeomorphism of class ${\mathcal C}^{1+\theta}$ uniformly in $t\in [0,T]$, where $\theta=1+\alpha-2/q$ if $q\neq2/\alpha$ and $\theta\in (0, 1)$ if $q=2/\alpha$. Additionally,
\begin{eqnarray}\label{3.5}
\frac{1}{2}<\sup_{0\leq t\leq T}\|\nabla\Phi(t)\|_0 <\frac{3}{2}
\quad {\rm and} \quad \frac{2}{3}<\sup_{0\leq t\leq T}\|\nabla\Psi(t)\|_0<2,
\end{eqnarray}
where $\Psi(t,\cdot)=\Phi^{-1}(t,\cdot)$.

In light  of It\^{o}'s formula (see \cite[Theorem 3.7]{KR}) and (\ref{3.2}), it follows that
\begin{eqnarray}\label{3.6}
\begin{split} dV(t,X_t)=&\partial_tV(t,X_t)dt+b(t,X_t)\cdot  \nabla V(t,X_t)dt + \frac{1}{2}\Delta V(t,X_t)dt+\nabla V(t,X_t)dB_t\\ =&\lambda V(t,X_t)dt-b(t,X_t)dt+\nabla V(t,X_t)dB_t
\end{split}
\end{eqnarray}
and
\begin{eqnarray}\label{3.7}
\begin{split}
dV(t,X_t^n)=&\partial_tV(t,X_t^n)dt+b(t,X_{\kappa_n(t)}^n)\cdot \nabla V(t,X_t^n)dt + \frac{1}{2}\Delta V(t,X_t^n)dt+\nabla V(t,X_t^n)dB_t \\
=&\lambda V(t,X_t^n)dt-b(t,X_t^n)dt+\nabla V(t,X_t^n)dB_t\\ &+[b(t,X_{\kappa_n(t)}^n)-b(t,X_t^n)]\cdot\nabla V(t,X_t^n)dt.
\end{split}
\end{eqnarray}
On account of (\ref{3.6}) and (\ref{3.7}), then
\begin{eqnarray}\label{3.8}
\begin{split} I_{t,n}=&V(t,X_t^n)-V(t,X_t)+\lambda \int_0^t[V(r,X_r)-V(r,X_r^n)]dr+\int_0^t[\nabla V(r,X_r)-\nabla V(r,X_r^n)]dB_r\\&-\int_0^t [b(r,X_{\kappa_n(r)}^n)-b(r,X_r^n)]\cdot \nabla V(r,X_r^n)dr.
\end{split}
\end{eqnarray}
For every $p\geq 2$, $\|I_{t,n}\|_{L^p(\Omega)}$ can be bounded above through (\ref{3.3}) and the Burkholder--Davis--Gundy inequality by
\begin{eqnarray*}
&&\|\nabla V\|_{\infty,0}\|X_t-X_t^n\|_{L^p(\Omega)}
+\lambda \|\nabla V\|_{\infty,0}\int_0^t\|X_r-X_r^n\|_{L^p(\Omega)}dr
\\&&
+C(p,d)\Big[{\mathbb E}\Big|\int_0^t|\nabla V(r,X_r)-\nabla V(r,X_r^n)|^2dr\Big|^{\frac{p}{2}}\Big]^{\frac{1}{p}}
\nonumber\\ &&+\Big\|\int_0^t [b(r,X_{\kappa_n(r)}^n)-b(r,X_r^n)]\cdot \nabla V(r,X_r^n)dr\Big\|_{L^p(\Omega)},
\end{eqnarray*}
and thereby concluding that
\begin{eqnarray}\label{3.9}
\begin{split}
\|I_{t,n}\|_{L^p(\Omega)} \leq&\frac{1}{2}\|X_t-X_t^n\|_{L^p(\Omega)}+\frac{\lambda}{2} \int_0^t\|X_r-X_r^n\|_{L^p(\Omega)}dr\\ &+C\Big[\int_0^t\|\nabla^2V(r)\|_0^2\|X_r-X_r^n\|_{L^p(\Omega)}^2dr\Big]^{\frac{1}{2}}
\\ &+\Big\|\int_0^t[b(r,X_{\kappa_n(r)}^n)-b(r,X_r^n)]\cdot \nabla V(r,X_r^n)dr\Big\|_{L^p(\Omega)},
\end{split}
\end{eqnarray}
if one uses Minkowski's inequality and (\ref{3.4}).

We use Lemma \ref{lem2.3} to estimate the last term in (\ref{3.9}). For this purpose, we set
\begin{eqnarray}\label{3.10}
A_{s,t}={\mathbb E}^{s_-}\int_s^t [b(r,B^x_{\kappa_n(r)})-b(r,B^x_r)]\cdot \nabla V(r,B^x_{s-})dr, \quad \forall \ (s,t)\in \overline{[0,T]}_\leq^2,
\end{eqnarray}
where $\{B^x_r\}_{r\in [0,T]}$ is a $d$-dimensional standard Brownian motion starting from $x$. Let's check conditions (\ref{2.7})--(\ref{2.9}) step by step.

Without loss of generality, we assume $s\in [mT/n,(m+1)T/n)$ for some $m\in \{0,1,\ldots,n-1\}$. Observe that if $(s,t)\in \overline{[0,T]}_\leq^2$, then $t\in [s,2(m+1)T/n]$.

If $t\in [s,(m+2)T/n]$, then
\begin{eqnarray}\label{3.11}
\begin{split}
\|A_{s,t}\|_{L^p(\Omega)}\leq& \|\nabla V\|_{\infty,0}\int_s^t \|b(r,B^x_{\kappa_n(r)})-b(r,B^x_r)\|_{L^p(\Omega)}dr \\ \leq& C_1(d,p)\int_s^t [b(r)]_\alpha|\kappa_n(r)-r|^{\frac{\alpha}{2}}dr \\ \leq& C_1(d,p,T)w_{b,\alpha,q}(s,t)^{\frac{1}{q}} (t-s)^{1-\frac{1}{q}}n^{-\frac{\alpha}{2}} \\ \leq& C_1(d,p,T)w_{b,\alpha,q}(s_-,t)^{\frac{1}{q}} (t-s_-)^{\frac{1}{2}-\frac{1}{q}+\varepsilon_1}n^{-\frac{1+\alpha}{2}+\varepsilon_1},
\end{split}
\end{eqnarray}
where $\varepsilon_1\in (0,1/2)$ is arbitrary and $w_{b,\alpha,q}(s,t)$ defined by (\ref{1.11}) is a control.

Otherwise, one has $t\in [(m+2)T/n,2(m+1)T/n]$. Hence,
\begin{eqnarray}\label{3.12}
\begin{split}
A_{s,t} =&{\mathbb E}^{s_-}\int_s^{\frac{(m+2)T}{n}}[b(r,B^x_{\kappa_n(r)})-b(r,B^x_r)]\cdot \nabla V(r,B^x_{s_-})dr
\\ & +\int_{\frac{(m+2)T}{n}}^t[({\mathcal P}_{\kappa_n(r)-s_-}-{\mathcal P}_{r-s_-})b(r,B^x_{s_-})]\cdot \nabla V(r,B^x_{s_-})dr=:A_{s,t}^1+A_{s,t}^2,
\end{split}
\end{eqnarray}
where $\{{\mathcal P}_t\}_{t\in [0,T]}$ is given by (\ref{1.6}) and in the second line we have used
\begin{eqnarray}\label{3.13}
{\mathbb E}^{s_-}[b(r,B^x_{\kappa_n(r)})-b(r,B^x_r)]
 ={\mathcal P}_{\kappa_n(r)-s_-}b(r,B^x_{s_-})-{\mathcal P}_{r-s_-}b(r,B^x_{s_-}).
\end{eqnarray}

For $A_{s,t}^1$, (\ref{3.11}) holds true, and for $A_{s,t}^2$ we have
\begin{eqnarray*}
\begin{split}
\|A_{s,t}^2\|_{L^p(\Omega)}
\leq& \frac{1}{2}
\int_{\frac{(m+2)T}{n}}^{t}dr\int_{\kappa_n(r)-s_-}^{r-s_-}\|\partial_\tau {\mathcal P}_\tau b(r)\|_0d\tau
\\ \leq& C_2(d,\alpha)
\int_{\frac{(m+2)T}{n}}^{t}[b(r)]_\alpha dr\int_{\kappa_n(r)-s_-}^{r-s_-}\tau^{\frac{\alpha}{2}-1}d\tau.
\end{split}
\end{eqnarray*}
Note that $\kappa_n(r)-s_-\leq r-s_- \leq 2(\kappa_n(r)-s_-)$, then
\begin{eqnarray}\label{3.14}
\begin{split}
\|A_{s,t}^2\|_{L^p(\Omega)} \leq&
C_2(d,\alpha)\int_{\frac{(m+2)T}{n}}^{t}[b(r)]_\alpha (r-s_-)^{-\frac{1}{2}+\varepsilon_1} dr\int_{\kappa_n(r)-s_-}^{r-s_-}\tau^{\frac{\alpha-1}{2}-\varepsilon_1}d\tau
 \\ \leq&C_2(d,\alpha,T)w_{b,\alpha,q}(s_-,t)^{\frac{1}{q}} (t-s_-)^{\frac{1}{2}-\frac{1}{q}+\varepsilon_1} n^{-\frac{1+\alpha}{2}+\varepsilon_1}.
\end{split}
\end{eqnarray}
By (\ref{3.11})--(\ref{3.14}), then
\begin{eqnarray}\label{3.15}
\|A_{s,t}\|_{L^p(\Omega)}\leq \Gamma_1w_1(s_-,t)^{\frac{1}{2}}(t-s)^{\varepsilon_1}, \quad \forall \ (s,t)\in \overline{[0,T]}_\leq^2,
\end{eqnarray}
where $\Gamma_1=[C_1(d,p,T)+C_2(d,\alpha,T)]n^{-\frac{1+\alpha}{2}+\varepsilon_1}$ and the control
\begin{eqnarray*}
w_1(s_-,t)=w_{b,\alpha,q}(s_-,t)^{\frac{2}{q}} (t-s_-)^{1-\frac{2}{q}}.
\end{eqnarray*}

For $(s,t)\in \overline{[0,T]}_\leq^2$ and $u=(s+t)/2$, we denote
$s_1:=s-(t-s)$, $s_2:=s-(u-s)$, $s_3:=s$, $s_4:=u$, $s_5:=t$. Then $s_1\leq s_2\leq s_3\leq s_4\leq s_5$ and
\begin{eqnarray}\label{3.16}
\begin{split}
&{\mathbb E}^{s_-}\delta A_{s,u,t}={\mathbb E}^{s_1}\delta A_{s_3,s_4,s_5}\\ =&{\mathbb E}^{s_1}\int_{s_3}^{s_4} [b(r,B^x_{\kappa_n(r)})-b(r,B^x_r)] \cdot [\nabla V(r,B^x_{s_1})-\nabla V(r,B^x_{s_2})]dr \\&+
{\mathbb E}^{s_1}\int_{s_4}^{s_5} [b(r,B^x_{\kappa_n(r)})-b(r,B^x_r)]\cdot[\nabla V(r,B^x_{s_1})-\nabla V(r,B^x_{s_3})]dr=: {\mathcal H}_{s,t,n}^1+{\mathcal H}_{s,t,n}^2.
\end{split}
 \end{eqnarray}
If $t-s<4T/n$, we obtain
\begin{eqnarray}\label{3.17}
\begin{split}
\|{\mathcal H}_{s,t,n}^1\|_{L^p(\Omega)} \leq& \int_{s_3}^{s_4}\|b(r,B^x_{\kappa_n(r)})-b(r,B^x_r)\|_{L^{2p}(\Omega)} \|\nabla V(r,B^x_{s_1})-\nabla V(r,B^x_{s_2})\|_{L^{2p}(\Omega)} dr \\ \leq& C_3(d,p,T)\int_{s_3}^{s_4}  [b(r)]_\alpha \|\nabla^2V(r) \|_0dr (s_2-s_1)^{\frac{1}{2}}n^{-\frac{\alpha}{2}}  \\ \leq& C_3(d,p,T) w_{b,\alpha,q}(s_-,t)^{\frac{1}{q}}\|\nabla^2V\|_{[s_-,t],q,0}^{\frac{1}{q}}(t-s_-)^{1-\frac{2}{q}} (t-s)^{\varepsilon_1}n^{-\frac{1+\alpha}{2}+\varepsilon_1}
\end{split}
\end{eqnarray}
and
\begin{eqnarray}\label{3.18}
\begin{split}
\|{\mathcal H}_{s,t,n}^2\|_{L^p(\Omega)} \leq& \int_{s_4}^{s_5} \|b(r,B^x_{\kappa_n(r)})-b(r,B^x_r)\|_{L^{2p}(\Omega)}\|\nabla V(r,B^x_{s_1})-\nabla V(r,B^x_{s_3})\|_{L^{2p}(\Omega)} dr\\ \leq& C_3(d,p,T) w_{b,\alpha,q}(s_-,t)^{\frac{1}{q}}\|\nabla^2V\|_{[s_-,t],q,0}^{\frac{1}{q}}(t-s_-)^{1-\frac{2}{q}} (t-s)^{\varepsilon_1}n^{-\frac{1+\alpha}{2}+\varepsilon_1},
\end{split}
\end{eqnarray}
where
\begin{eqnarray*}
\|\nabla^2V\|_{[s_-,t],q,0}=\int_{s_-}^t\|\nabla^2V(r)\|^q_0dr.
\end{eqnarray*}
If $t-s\geq 4T/n$, then
\begin{eqnarray}\label{3.19}
\begin{split}
\|{\mathcal H}_{s,t,n}^1\|_{L^p(\Omega)} \leq&
C_3\int_{s_3}^{s_4}\|\nabla^2V(r)\|_0(s_2-s_1)^{\frac{1}{2}}dr
\int_{\kappa_n(r)-s_2}^{r-s_2}\|\partial_\tau{\mathcal P}_\tau b(r)\|_0d\tau
 \\ \leq&
C_3\int_{s_3}^{s_4}  [b(r)]_\alpha \|\nabla^2V(r) \|_0(s_2-s_1)^{\frac{1}{2}}dr \int_{\kappa_n(r)-s_2}^{r-s_2}\tau^{\frac{\alpha}{2}-1}d\tau
\\ \leq&
C_3\int_{s_3}^{s_4}  [b(r)]_\alpha \|\nabla^2V(r) \|_0 (s_2-s_1)^{\frac{1}{2}}(\kappa_n(r)-s_2)^{-\frac{1}{2}+\varepsilon_1}dr n^{-\frac{1+\alpha}{2}+\varepsilon_1}.
\end{split}
\end{eqnarray}
For every $r\in [s_3,s_5]$ and every $\varepsilon_1\in (0,1/2)$,
$$
(s_2-s_1)^{\frac{1}{2}}(\kappa_n(r)-s_2)^{-\frac{1}{2}+\varepsilon_1}\leq [2^{-1}(t-s)]^{\frac{1}{2}}[4^{-1}(t-s)] ^{-\frac{1}{2}+\varepsilon_1}\leq \sqrt{2}(t-s)^{\varepsilon_1},
$$
where in the first inequality we have used
\begin{eqnarray*}
\kappa_n(r)\geq s_3-T/n\geq s-(t-s)/4=s-(u-s)+(t-s)/4=s_2+(t-s)/4>s_2.
\end{eqnarray*}
This, together with (\ref{3.19}), leads to
\begin{eqnarray}\label{3.20}
\|{\mathcal H}_{s,t,n}^1\|_{L^p(\Omega)} \leq C_3 w_{b,\alpha,q}(s_-,t)^{\frac{1}{q}}\|\nabla^2V\|_{[s_-,t],q,0}^{\frac{1}{q}}(t-s_-)^{1-\frac{2}{q}} (t-s)^{\varepsilon_1}n^{-\frac{1+\alpha}{2}+\varepsilon_1}.
\end{eqnarray}
For ${\mathcal H}_{s,t,n}^2$, we have
\begin{eqnarray*}
\|{\mathcal H}_{s,t,n}^2\|_{L^p(\Omega)} \leq
C_3\int_{s_4}^{s_5}  [b(r)]_\alpha \|\nabla^2V(r) \|_0 (s_3-s_1)^{\frac{1}{2}}(\kappa_n(r)-s_3)^{-\frac{1}{2}+\varepsilon_1}dr n^{-\frac{1+\alpha}{2}+\varepsilon_1}.
\end{eqnarray*}
Note that for every $r\in [s_4,s_5]$ and every $\varepsilon_1\in (0,1/2)$,
$$
(s_3-s_1)^{\frac{1}{2}}(\kappa_n(r)-s_3)^{-\frac{1}{2}+\varepsilon_1}\leq (t-s)^{\frac{1}{2}}[4^{-1}(t-s)] ^{-\frac{1}{2}+\varepsilon_1}\leq 2(t-s)^{\varepsilon_1}
$$
because of the shifted basepoint; in general, this would not be true with $s_-$ replaced by $s$ in (\ref{3.10}) since now the following estimate is not true for $r\in [s_4,s_5]$
$$
(s_4-s_3)^{\frac{1}{2}}(r-s_4)^{-\frac{1}{2}+\varepsilon_1}
\leq C(t-s)^{\varepsilon_1}.
$$
Hence,
\begin{eqnarray}\label{3.21}
\|{\mathcal H}_{s,t,n}^2\|_{L^p(\Omega)}\leq C_3 w_{b,\alpha,q}(s_-,t)^{\frac{1}{q}}\|\nabla^2V\|_{[s_-,t],q,0}^{\frac{1}{q}}(t-s_-)^{1-\frac{2}{q}} (t-s)^{\varepsilon_1}n^{-\frac{1+\alpha}{2}+\varepsilon_1}.
\end{eqnarray}
In view of (\ref{3.16})--(\ref{3.18}), (\ref{3.20}) and (\ref{3.21}), we conclude
\begin{eqnarray}\label{3.22}
\|{\mathbb E}^{s_-}\delta A_{s,u,t}\|_{L^p(\Omega)}\leq \Gamma_2 w_2(s_-,t)(t-s)^{\varepsilon_1},
\end{eqnarray}
with $\Gamma_2=C_3(p,d,T)n^{-\frac{1+\alpha}{2}+\varepsilon_1}$ and
$$
w_2(s_-,t)=w_{b,\alpha,q}(s_-,t)^{\frac{1}{q}}\|\nabla^2V\|_{[s_-,t],q,0}^{\frac{1}{q}}
(t-s_-)^{1-\frac{2}{q}},
$$
which is a control since $w_{b,\alpha,q}(s_-,t)$, $\|\nabla^2V\|_{[s_-,t],q,0}$ and $(t-s_-)$ are controls.

Now, let us check (\ref{2.9}). For $(s,t)\in [0,T]_\leq^2$ and $k\in {\mathbb N}$, we denote $t_i:=s+i(t-s)/k$, $i=0,1,\ldots,k$, then $t_i-(t_{i+1}-t_i)=t_{i-1}$. Set
\begin{eqnarray}\label{3.23}
{\mathcal A}_t=\int_0^t [b(r,B^x_{\kappa_n(r)})-b(r,B^x_r)] \cdot\nabla V(r,B^x_r) dr
\end{eqnarray}
and
\begin{eqnarray*}
\left\{\begin{array}{ll}\mathcal{E}_{t_{i-1},n}^1(r):= [b(r,B^x_{\kappa_n(r)})-b(r,B^x_r)]\cdot[\nabla V(r,B^x_r)-\nabla V(r,B^x_{t_{i-1}})],\\ [0.2cm]
\mathcal{E}_{t_{i-1},n}^2(r):=[b(r,B^x_{\kappa_n(r)})-b(r,B^x_r)-{\mathbb E}^{t_{i-1}}(b(r,B^x_{\kappa_n(r)})-b(r,B^x_r))] \cdot \nabla V(r,B^x_{t_{i-1}}) , \\ [0.2cm] \mathcal{E}_n^3(r):=[b(r,B^x_{\kappa_n(r)})-b(r,B^x_r)]\cdot \nabla V(r,B^x_r),
\end{array}\right.
\end{eqnarray*}
then
\begin{eqnarray}\label{3.24}
\begin{split}
{\mathcal A}_t-{\mathcal A}_s-\sum_{i=1}^{k-1}A_{t_i,t_{i+1}} =\sum_{i=1}^{k-1}\int_{t_i}^{t_{i+1}}\big[\mathcal{E}_{t_{i-1},n}^1(r)+\mathcal{E}_{t_{i-1},n}^2(r)\big]dr
+\int_{t_0}^{t_1}\mathcal{E}_n^3(r)dr.
\end{split}
\end{eqnarray}

Due to H\"{o}lder's inequality, (\ref{3.4}) and (\ref{1.9}),
\begin{eqnarray*}
\begin{split}
\|\mathcal{E}_{t_{i-1},n}^1(r)\|_{L^1(\Omega)} \leq& [b(r)]_\alpha \|\nabla^2 V(r)\|_0{\mathbb E}[|B^x_{\kappa_n(r)}-B^x_r|^\alpha |B^x_r-B^x_{t_{i-1}}|] \\  \leq& C(d,\alpha)[b(r)]_\alpha \|\nabla^2 V(r)\|_0|\kappa_n(r)-r|^{\frac{\alpha}{2}}(r-t_{i-1})^{\frac{1}{2}}
\end{split}
\end{eqnarray*}
and
\begin{eqnarray*}
\begin{split}\|\mathcal{E}_{t_{i-1},n}^2(r)\|_{L^1(\Omega)} \leq& {\mathbb E}|b(r,B^x_r)-b(r,B^x_{t_{i-1}})|+ 1_{\{\kappa_n(r)\geq t_{i-1}\}} {\mathbb E}|b(r,B^x_{\kappa_n(r)})-b(r,B^x_{t_{i-1}})| \\  \leq& C(d,\alpha)[b(r)]_\alpha \big[ (r-t_{i-1})^{\frac{\alpha}{2}}+ 1_{\{\kappa_n(r)\geq t_{i-1}\}}|\kappa_n(r)-t_{i-1}|^{\frac{\alpha}{2}} \big],
\end{split}
\end{eqnarray*}
which lead to
\begin{eqnarray}\label{3.25}
\begin{split}
&\sum_{i=1}^{k-1}\int_{t_i}^{t_{i+1}}\big[\|\mathcal{E}_{t_{i-1},n}^1(r)\|_{L^1(\Omega)}+\|\mathcal{E}_{t_{i-1},n}^2(r)\|_{L^1(\Omega)}\big] dr\\ \leq& C\int_s^t[b(r)]_\alpha \|\nabla^2 V(r)\|_0drn^{-\frac{\alpha}{2}}k^{-\frac{1}{2}}+C\int_s^t[b(r)]_\alpha drk^{-\frac{\alpha}{2}}.
\end{split}
\end{eqnarray}

For $\mathcal{E}_n^3(r)$, we have
\begin{eqnarray}\label{3.26}
\int_{t_i}^{t_{i+1}}\|\mathcal{E}_{n}^3(r)\|_{L^1(\Omega)}dr\leq \|b\|_{q,0}k^{\frac{1}{q}-1}.
\end{eqnarray}
Combing (\ref{3.24})--(\ref{3.26}), one derives
\begin{eqnarray}\label{3.27}
\Big\|{\mathcal A}_t-{\mathcal A}_s-\sum_{i=1}^{k-1}A_{t_i,t_{i+1}}\Big\|_{L^1(\Omega)}
 \leq C(d,\alpha,T,\|b\|_{q,\alpha})k^{-\frac{\alpha}{2}}\rightarrow 0, \ \ {\rm as} \ \ k\rightarrow\infty.
\end{eqnarray}
(\ref{3.15}), (\ref{3.22}) and (\ref{3.27}), together with Lemma \ref{lem2.3}, imply that the stochastic process defined by (\ref{3.23}) is $L^p$-integrable which satisfies (\ref{2.10}), i.e.,
\begin{eqnarray}\label{3.28}
\begin{split}
\|{\mathcal A}_{s,t}\|_{L^p(\Omega)} \leq&
C\Big[1+\|\nabla^2V\|_{[s,t],q,0}^{\frac{1}{q}}
(t-s)^{\frac{1}{2}-\frac{1}{q}}\Big]w_{b,\alpha,q}(s,t)^{\frac{1}{q}}(t-s)^{\frac{1}{2}-\frac{1}{q}+\varepsilon_1}
 n^{-\frac{1+\alpha}{2}+\varepsilon_1}\\ \leq&
Cw_{b,\alpha,q}(s,t)^{\frac{1}{q}}(t-s)^{\frac{1}{2}-\frac{1}{q}+\varepsilon_1}n^{-\frac{1+\alpha}{2}+\varepsilon_1}, \quad \forall \ (s,t)\in [0,T]_\leq^2,
\end{split}
\end{eqnarray}
where  ${\mathcal A}_{s,t}={\mathcal A}_t-{\mathcal A}_s$ and the constant $C$ depends only on $d,\alpha,T,p,q$ and $\|b\|_{q,\alpha}$.

Notice that
\begin{eqnarray*}
X_t^n=x+\int_0^tb(r,X_{\kappa_n(r)}^n)dr+B_t=\int_0^tb(r,X_{\kappa_n(r)}^n)dr+B^x_t,
\end{eqnarray*}
where $\{B_t\}_{t\in [0,T]}$ is a $d$-dimensional standard Brownian motion starting from $0$. By using the Girsanov theorem, $\{X_t^n\}_{t\in[0,T]}$ is a standard Brownian motion starting from $x$ on the stochastic
basis ($\Omega, {\mathcal F},\{{\mathcal F}_t\}_{t\in[0,T]},{\mathbb Q}$), where
\begin{eqnarray*}
\frac{d{\mathbb Q}}{d{\mathbb P}}=\exp\Big(-\int_0^tb(r,X_{\kappa_n(r)}^n)\cdot dB^x_r-\frac{1}{2}\int_0^t|b(r,X_{\kappa_n(r)}^n)|^2dr\Big).
\end{eqnarray*}
Thus,
\begin{eqnarray*}
\begin{split}
&{\mathbb E}\Big|\int_s^t[b(r,X_{\kappa_n(r)}^n)-b(r,X_r^n)] \cdot\nabla V(r,X_r^n) dr\Big|^p
\\
=&{\mathbb E}_{\mathbb Q}\Big[\Big|\int_s^t[b(r,X_{\kappa_n(r)}^n)-b(r,X_r^n)]\cdot \nabla V(r,X_r^n)dr\Big|^p\frac{d\mathbb{P}}{d\mathbb{Q}}\Big]
\\ =& {\mathbb E}\Big[|{\mathcal A}_{s,t}|^p\exp\Big(\int_0^tb(r,B^x_{\kappa_n(r)})\cdot dB^x_r-\frac{1}{2}\int_0^t|b(r,B^x_{\kappa_n(r)})|^2dr
\Big)
\Big].
\end{split}
\end{eqnarray*}
This, together with (\ref{3.28}) and H\"{o}lder's inequality, leads to
\begin{eqnarray}\label{3.29}
\begin{split}
&{\mathbb E}\Big|\int_s^t[b(r,X_{\kappa_n(r)}^n)-b(r,X_r^n)] \cdot\nabla V(r,X_r^n) dr\Big|^p
\\
\leq &\|{\mathcal A}_{s,t}\|_{L^{2p}(\Omega)}^p\exp\Big(\frac{1}{2}\int_0^t\|b(r)\|_0^2dr\Big)
\\ &\times \Big[{\mathbb E}\exp\Big(2\int_0^tb(r,B^x_{\kappa_n(r)})\cdot dB^x_r-2\int_0^t|b(r,B^x_{\kappa_n(r)})|^2dr\Big)\Big]^{\frac{1}{2}}
\\ \leq &
C\exp\Big(\frac{1}{2p}\int_0^t\|b(r)\|_0^2dr\Big)w_{b,\alpha,q}(s,t)^{\frac{1}{q}}(t-s)^{\frac{1}{2}-\frac{1}{q}+\varepsilon_1}n^{-\frac{1+\alpha}{2}+\varepsilon_1}.
\end{split}
\end{eqnarray}

Owing to (\ref{3.9}) and (\ref{3.29}), it follows that
\begin{eqnarray}\label{3.30}
\begin{split}
\|I_{t,n}\|_{L^p(\Omega)}  \leq&\frac{1}{2}\|X_t-X_t^n\|_{L^p(\Omega)}+\frac{\lambda}{2} \int_0^t\|X_r-X_r^n\|_{L^p(\Omega)}dr\\ &+C\Big[\int_0^t\|\nabla^2V(r)\|_0^2\|X_r-X_r^n\|_{L^p(\Omega)}^2dr\Big]^{\frac{1}{2}}\\ &+Cw_{b,\alpha,q}(0,t)^{\frac{1}{q}}
t^{\frac{1}{2}-\frac{1}{q}+\varepsilon_1}n^{-\frac{1+\alpha}{2}+\varepsilon_1}, \quad \forall \ t\in [0,T].
\end{split}
\end{eqnarray}

\textbf{Step 2:} To estimate $J_{t,n}$. Replacing $\nabla V$ by $I_{d\times d}$ in (\ref{3.10}), and using calculations similar to those in \textbf{Step 1}  (easier here since ${\mathbb E}^{s-}\delta A_{s,u,t}=0$), for every $(s,t)\in [0,T]_\leq^2$, one finds that
\begin{eqnarray}\label{3.31}
\quad \Big\|\int_s^t[b(r,X_{\kappa_n(r)}^n)-b(r,X_r^n)]dr
\Big\|_{L^p(\Omega)}
\leq
Cw_{b,\alpha,q}(s,t)^{\frac{1}{q}}(t-s)^{\frac{1}{2}-\frac{1}{q}
+\varepsilon_1}n^{-\frac{1+\alpha}{2}+\varepsilon_1}.
\end{eqnarray}

\textbf{Step 3:}  To estimate $X_t-X_t^n$. Summing over (\ref{3.1}), (\ref{3.30}) and (\ref{3.31}), it follows that
\begin{eqnarray*}
\begin{split} \|X_t-X_t^n\|_{L^p(\Omega)} \leq&
\frac{1}{2}\|X_t-X_t^n\|_{L^p(\Omega)}+C\Big[\int_0^t\|\nabla^2V(r)\|_0^2\|X_r-X_r^n\|_{L^p(\Omega)}^2dr\Big]^{\frac{1}{2}}
\\ &
+\frac{\lambda}{2} \int_0^t\|X_r-X_r^n\|_{L^p(\Omega)}dr+
Cw_{b,\alpha,q}(0,t)^{\frac{1}{q}}t^{\frac{1}{2}-\frac{1}{q}
+\varepsilon_1}n^{-\frac{1+\alpha}{2}+\varepsilon_1}.
\end{split}
\end{eqnarray*}
Whence,
\begin{eqnarray*}
\|X_t-X_t^n\|_{L^p(\Omega)}^2 \leq C\int_0^t\big[1+\|\nabla^2V(r)\|_0^2\big]\|X_r-X_r^n\|_{L^p(\Omega)}^2dr+
Ct^{1-\frac{2}{q}+2\varepsilon_1}n^{-1-\alpha+2\varepsilon_1},
\end{eqnarray*}
which implies
\begin{eqnarray}\label{3.32}
\|X-X^n\|_{t,p}\leq Ct^{\frac{1}{2}-\frac{1}{q}
+\varepsilon_1}n^{-\frac{1+\alpha}{2}+\varepsilon_1}, \quad  \forall \ t\in [0,T],
\end{eqnarray}
if one uses Gr\"{o}nwall's inequality, where the notation $\|\cdot\|_{t,p}$ is given by (\ref{1.10}).

\smallskip
\textbf{Step 4:}  To estimate $\sup\limits_{0\leq t\leq T}|X_t-X_t^n|$. Let $I_{\cdot,n}$ and $J_{\cdot,n}$ be given in (\ref{3.1}). For $(s,t)\in [0,T]_\leq^2$, set $I_{s,t,n}=I_{t,n}-I_{s,n}$ and $J_{s,t,n}=J_{t,n}-J_{s,n}$.
Taking (\ref{3.6}) and (\ref{3.7}) into consideration, it leads to
\begin{eqnarray}\label{3.33}
\begin{split}
I_{s,t,n}=&V(t,X_t^n)-V(t,X_t)-V(s,X_s^n)+V(s,X_s) \\ & +\lambda \int_s^t[V(r,X_r)-V(r,X_r^n)]dr  +\int_s^t[\nabla V(r,X_r)-\nabla V(r,X_r^n)]dB_r\\ &-\int_s^t[b(r,X_{\kappa_n(r)}^n)-b(r,X_r^n)] \cdot \nabla V(r,X_r^n)dr.
\end{split}
\end{eqnarray}
We rewrite the difference of $V(\cdot,X_\cdot^n)-V(\cdot,X_\cdot)$ at times $t$ and $s$ by
\begin{eqnarray}\label{3.34}
\begin{split}
&\big[V(t,X_t^n)-V(t,X_t)-V(t,X_s^n)+V(t,X_s)\big]\\ & +\big[V(t,X_s^n)-V(t,X_s)-V(s,X_s^n)+V(s,X_s)\big]. \
\end{split}
\end{eqnarray}
On account of (\ref{2.19}), (\ref{3.3}) and (\ref{3.4}), for every $\theta\in (0,(1+\alpha-2/q)\wedge 1)$ then
\begin{eqnarray}\label{3.35}
\begin{split}
& |V(t,X_t^n)-V(t,X_t)-V(t,X_s^n)+V(t,X_s)|
\\ \leq& \|\nabla V\|_{\infty,0} |X_t-X_t^n-X_s+X_s^n|\\ &+ \frac{1}{1+\theta}[\nabla V]_{\infty,\theta}\big[|X_t-X_s|^\theta+|X_t^n-X_s^n|^\theta\big]|X_t-X_t^n|
\\ \leq& \frac{1}{2}|X_t-X_t^n-X_s+X_s^n|+C\big[|X_t-X_s|^\theta+|X_t^n-X_s^n|^\theta\big]|X_t-X_t^n|.
\end{split}
\end{eqnarray}
For the second term in (\ref{3.34}), its absolute can be estimated, by (\ref{3.3}), that
\begin{eqnarray}\label{3.36}
\begin{split}
&\Big|\int_0^1\big[\nabla V(t,X_s+r(X_s^n-X_s))-\nabla V(s,X_s+r(X_s^n-X_s))]\cdot(X_s^n-X_s)dr\Big|
\\ \leq& \|\nabla V\|_{{\mathcal C}^{\frac{\theta}{2}}_b([0,T];{\mathcal C}_b({\mathbb R}^d))}|X_s^n-X_s|(t-s)^{\frac{\theta}{2}}\leq C|X_s^n-X_s|(t-s)^{\frac{\theta}{2}},
\end{split}
\end{eqnarray}
where the parameter $\theta$ is given in (\ref{3.35}).

Summing over (\ref{3.33})--(\ref{3.36}), for every $p\geq 2$,  $\|I_{s,t,n}\|_{L^p(\Omega)}$ admits an upper bound of
\begin{eqnarray}\label{3.37}
\begin{split} &\frac{1}{2}\|X_t-X_t^n-X_s+X_s^n\|_{L^p(\Omega)}+C\Big\|\big[|X_t^n-X_s^n|^\theta+|X_t-X_s|^\theta\big]|X_t-X_t^n|\Big\|_{L^p(\Omega)}
\\ & +C\|X_s^n-X_s\|_{L^p(\Omega)}(t-s)^{\frac{\theta}{2}}+\frac{\lambda}{2} \int_s^t\|X_r-X_r^n\|_{L^p(\Omega)}dr
\\ &
+C\Big[\int_s^t\|\nabla^2V(r)\|_0^2\|X_r-X_r^n\|_{L^p(\Omega)}^2dr\Big]^{\frac{1}{2}}
\\ &
+\Big\|\int_s^t\nabla V(r,X_r^n)\cdot [b(r,X_{\kappa_n(r)}^n)-b(r,X_r^n)]dr\Big\|_{L^p(\Omega)}.
\end{split}
\end{eqnarray}
Thanks to H\"{o}lder's inequality and (\ref{3.32}), then
\begin{eqnarray}\label{3.38}
\begin{split} &\Big\|\big[|X_t^n-X_s^n|^\theta+|X_t-X_s|^\theta\big]|X_t-X_t^n|
\Big\|_{L^p(\Omega)}
\\ \leq &
\big[\|X_t^n-X_s^n\|_{L^{2p}(\Omega)}^\theta+\|X_t-X_s\|_{L^{2p}(\Omega)}^\theta\big]
\|X_t-X_t^n\|_{L^{2p}(\Omega)}
\\ \leq & C\big[\|X_t^n-X_s^n\|_{L^{2p}(\Omega)}^\theta+
\|X_t-X_s\|_{L^{2p}(\Omega)}^\theta\big]n^{-\frac{1+\alpha}{2}+\varepsilon_1}.
\end{split}
\end{eqnarray}
For the fourth term in (\ref{3.37}), by (\ref{3.32}) we calculate that
\begin{eqnarray}\label{3.39}
\frac{\lambda}{2} \int_s^t\|X_r-X_r^n\|_{L^p(\Omega)}dr\leq C \int_s^tr^{\frac{1}{2}-\frac{1}{q}
+\varepsilon_1} dr n^{-\frac{1+\alpha}{2}+\varepsilon_1} \leq C(t-s)n^{-\frac{1+\alpha}{2}+\varepsilon_1}.
\end{eqnarray}
Combining (\ref{3.29}), (\ref{3.32}), (\ref{3.33}), (\ref{3.37})--(\ref{3.39}), we conclude
\begin{eqnarray}\label{3.40}
\begin{split}
&\|X_t-X_t^n-X_s+X_s^n\|_{L^p(\Omega)}=\|I_{s,t,n}+J_{s,t,n}\|_{L^p(\Omega)}\\
\leq& C\Bigg\{\|X_t^n-X_s^n\|_{L^{2p}(\Omega)}^\theta+\|X_t-X_s\|_{L^{2p}(\Omega)}^\theta+
\Big[\int_s^t\|\nabla^2V(r)\|_0^2r^{1-\frac{2}{q}
+2\varepsilon_1}dr\Big]^{\frac{1}{2}}\\ & +
(t-s)^{\frac{\theta}{2}}+w_{b,\alpha,q}(s,t)^{\frac{1}{q}}(t-s)^{\frac{1}{2}-\frac{1}{q}
+\varepsilon_1}\Bigg\}
n^{-\frac{1+\alpha}{2}+\varepsilon_1},
\end{split}
\end{eqnarray}
where $\varepsilon_1\in (0,1/2)$ is arbitrary.

Since $X_t$ and $X_t^n$ satisfy (\ref{1.1}) and (\ref{1.3}), respectively, we obtain
\begin{eqnarray}\label{3.41}
\begin{split}
\|X_t-X_s\|_{L^{2p}(\Omega)}+\|X_t^n-X_s^n\|_{L^{2p}(\Omega)}\leq C(t-s)^{\frac{1}{2}}.
\end{split}
\end{eqnarray}
By (\ref{3.3}), we obtain
\begin{eqnarray*}
\left\{\begin{array}{ll}\nabla V\in L^\infty([0,T];{\mathcal C}_b^{1+\alpha-\frac{2}{q}}({\mathbb R}^d;{\mathbb R}^{d\times d})), & {\rm if} \  \ q\in (2/\alpha,\infty],\\ [0.2cm]
\nabla V\in
\mathop{\cap}\limits_{\tilde{q}\in [q, \infty)}L^{\tilde{q}}([0,T];{\mathcal C}_b^1({\mathbb R}^d;{\mathbb R}^{d\times d})), \quad & {\rm if} \  \ q=2/\alpha, \\ [0.2cm]
 \nabla V\in  L^{\frac{2q}{2-q\alpha}}([0,T];{\mathcal C}_b^1({\mathbb R}^d;{\mathbb R}^{d\times d})), & {\rm if} \  \ q\in [2,2/\alpha).
\end{array}\right.
\end{eqnarray*}
This implies
\begin{eqnarray}\label{3.42}
\int_s^t\|\nabla^2V(r)\|_0^2r^{1-\frac{2}{q}
+2\varepsilon_1}dr \leq \left\{\begin{array}{ll} C(t-s)^{(1+\alpha-2/q)\wedge 1}, \quad & {\rm if} \  \ q\neq 2/\alpha,
\\ [0.2cm] C(t-s)^{1-\epsilon_0}, & {\rm if} \  \ q=2/\alpha,
\end{array}\right.
\end{eqnarray}
where $\epsilon_0\in (0,1)$ is sufficiently small.

By virtue of (\ref{3.40})--(\ref{3.42}), there is a constant $C=C(d,\alpha,T,p,q,\varepsilon_1,\|b\|_{q,\alpha})$ such that
\begin{eqnarray}\label{3.43}
\begin{split}
\|X_t-X_t^n-X_s+X_s^n\|_{L^p(\Omega)}\leq& C
\Big[(t-s)^{\frac{\theta}{2}}+(t-s)^{\frac{1}{2}-\frac{1}{q}
+\varepsilon_1}\Big]
n^{-\frac{1+\alpha}{2}+\varepsilon_1}\\ \leq& C
(t-s)^{\frac{1}{2}-\frac{1}{q}+\varepsilon_1}
n^{-\frac{1+\alpha}{2}+\varepsilon_1}, \quad \forall \ p\in [2,\infty),
\end{split}
\end{eqnarray}
if one chooses $\varepsilon_1<\alpha/2<1/2$ and $\theta>1-2/q+2\varepsilon_1$.  By appealing to Lemma \ref{lem2.5} with $\beta=0,m=p$ and H\"{o}lder's inequality, the estimate (\ref{1.16}) holds.

\section{Proof of Theorem \ref{the1.3}}
\label{sec4}\setcounter{equation}{0}
Let $X_t$ and $X_t^n$ be given by (\ref{1.1}) and (\ref{1.3}). We denote $\varphi_t=X_t-B_t$ and $\varphi^n_t=X^n_t-B_t$. For every $n\in {\mathbb N}$ and $t\in [0,T]$, we get an analogue of (\ref{3.1}) that
\begin{eqnarray}\label{4.1}
\begin{split} &X_t-X_t^n\\ \quad =&
\int_0^t[b(r,B_r+\varphi_r)-b(r,B_r+\varphi_r^n)]dr+\int_0^t[b(r,B_r+\varphi_r^n)-
b(r,B_r+\varphi_{\kappa_n(r)}^n)]dr\\ &+ \int_0^t[b(r,B_r+\varphi_{\kappa_n(r)}^n)-
b(r,B_{\kappa_n(r)}+\varphi_{\kappa_n(r)}^n)]dr =:{\mathcal I}_{t,n}+{\mathcal J}_{t,n}+{\mathcal K}_{t,n}.
\end{split}
\end{eqnarray}
We divide the calculations into four steps, and throughout these steps, we assume $(S_0,T_0)\in [0,T]_\leq^2$.

\textbf{Step 1:} To estimate ${\mathcal I}_{t,n}$. For $(s,t)\in \overline{[S_0,T_0]}_\leq^2$, set
\begin{eqnarray}\label{4.2}
A_{s,t}={\mathbb E}^{s_-}\int_s^t [b(r,B_r+{\mathbb E}^{s_-}\varphi_r)-b(r,B_r+{\mathbb E}^{s_-}\varphi_r^n)]dr.
\end{eqnarray}
Let $\{{\mathcal P}_t\}_{t\in [0,T]}$ be given by (\ref{1.6}). Then
\begin{eqnarray}\label{4.3}
\begin{split} |A_{s,t}|\leq& \int_s^t | {\mathcal P}_{r-s_-}b(r,B_{s_-}+{\mathbb E}^{s_-}\varphi_r)-{\mathcal P}_{r-s_-}b(r,B_{s_-}+{\mathbb E}^{s_-}\varphi_r^n)|dr\\ \leq& \int_s^t \|\nabla {\mathcal P}_{r-s_-}b(r)\|_0|{\mathbb E}^{s_-}(\varphi_r-\varphi_r^n)|dr\\ \leq& C(d,\alpha)\int_s^t [b(r)]_\alpha (r-s_-)^{\frac{\alpha-1}{2}}|{\mathbb E}^{s_-}(\varphi_r-\varphi_r^n)|dr.
\end{split}
\end{eqnarray}
For any $p\in [2,\infty)$, it follows that
\begin{eqnarray}\label{4.4}
\begin{split} \|A_{s,t}\|_{L^p(\Omega)}
\leq C(d,\alpha)w_{b,\alpha,q}(s_-,t)^{\frac{1}{q}}(t-s)^{\frac{1+\alpha}{2}-\frac{1}{q}}\|\varphi-\varphi^n\|_{[s,t],p}.
\end{split}
\end{eqnarray}

Let $u$ and $s_i, \ i=1,\ldots,5$ be defined in (\ref{3.16}). Then
\begin{eqnarray}\label{4.5}
\begin{split}
&{\mathbb E}^{s_-}\delta A_{s,u,t}={\mathbb E}^{s_1}\delta A_{s_3,s_4,s_5}\\ =&{\mathbb E}^{s_1}\int_{s_3}^{s_4}[b(r,B_r+{\mathbb E}^{s_1}\varphi_r)-b(r,B_r+{\mathbb E}^{s_1}\varphi_r^n)-b(r,B_r+{\mathbb E}^{s_2}
\varphi_r)\\& +b(r,B_r+{\mathbb E}^{s_2}\varphi_r^n)]dr+
{\mathbb E}^{s_1}\int_{s_4}^{s_5}[b(r,B_r+{\mathbb E}^{s_1}\varphi_r)-b(r,B_r+{\mathbb E}^{s_1}\varphi_r^n)\\& -b(r,B_r+{\mathbb E}^{s_3}
\varphi_r)+b(r,B_r+{\mathbb E}^{s_3}\varphi_r^n)]dr
 =: {\mathcal I}_{s,t,n}^1+{\mathcal I}_{s,t,n}^2.
\end{split}
\end{eqnarray}
We begin with the analysis of ${\mathcal I}_{s,t,n}^1$. Recalling the heat semigroup in (\ref{4.3}) and the tower property for the conditional expectation, ${\mathcal I}_{s,t,n}^1$ can be rewritten as
\begin{eqnarray}\label{4.6}
\begin{split}&{\mathbb E}^{s_1}{\mathbb E}^{s_2}\int_{s_3}^{s_4}[b(r,B_r+{\mathbb E}^{s_1}\varphi_r)-b(r,B_r+{\mathbb E}^{s_1}\varphi_r^n)
-b(r,B_r+{\mathbb E}^{s_2}
\varphi_r)+b(r,B_r+{\mathbb E}^{s_2}\varphi_r^n)]dr
\\
 =&{\mathbb E}^{s_1}\int_{s_3}^{s_4}[{\mathcal P}_{r-s_2}b(r,B_{s_2}+{\mathbb E}^{s_1}\varphi_r)-
 {\mathcal P}_{r-s_2}b(r,B_{s_2}+{\mathbb E}^{s_1}
 \varphi_r^n)\\
 &\qquad\quad \ \ -{\mathcal P}_{r-s_2}b(r,B_{s_2}+{\mathbb E}^{s_2}
\varphi_r)+{\mathcal P}_{r-s_2}b(r,B_{s_2}+{\mathbb E}^{s_2}\varphi_r^n)]dr.
\end{split}
\end{eqnarray}
We choose $f(r,x)={\mathcal P}_{r-s_2}b(r,x)$ and take $x_1=B_{s_2}+{\mathbb E}^{s_1}\varphi_r$, $x_2=B_{s_2}+{\mathbb E}^{s_1}\varphi_r^n$, $x_3=B_{s_2}+{\mathbb E}^{s_2}
\varphi_r$, $x_4=B_{s_2}+{\mathbb E}^{s_2}\varphi_r^n$ in Lemma \ref{lem2.6}. By (\ref{2.18}), $|{\mathcal I}_{s,t,n}^1|$ is bounded above by
\begin{eqnarray*}
&&\int_{s_3}^{s_4}\|\nabla {\mathcal P}_{r-s_2}b(r)\|_0{\mathbb E}^{s_1}|{\mathbb E}^{s_1}(\varphi_r-\varphi_r^n)-{\mathbb E}^{s_2}
(\varphi_r-\varphi_r^n)|dr
+\frac{1}{2}\int_{s_3}^{s_4}\|\nabla^2 {\mathcal P}_{r-s_2}b(r)\|_0
\\ &&\quad \times|{\mathbb E}^{s_1}(\varphi_r-\varphi_r^n)|\big[{\mathbb E}^{s_1}|{\mathbb E}^{s_1}\varphi_r-{\mathbb E}^{s_2}
\varphi_r|+{\mathbb E}^{s_1}|{\mathbb E}^{s_1}\varphi_r^n-{\mathbb E}^{s_2}
\varphi_r^n|\big]dr.
\end{eqnarray*}
Observing that $s_1\leq s_2$ and
\begin{eqnarray*}
\begin{split}
{\mathbb E}^{s_1}|{\mathbb E}^{s_1}(\varphi_r-\varphi_r^n)-{\mathbb E}^{s_2}
(\varphi_r-\varphi_r^n)|=&{\mathbb E}^{s_1}|{\mathbb E}^{s_2}[\varphi_r-\varphi_r^n-{\mathbb E}^{s_1}
(\varphi_r-\varphi_r^n)]|\\
 \leq& {\mathbb E}^{s_1}|\varphi_r-\varphi_r^n-{\mathbb E}^{s_1}
(\varphi_r-\varphi_r^n)|,
\end{split}
\end{eqnarray*}
we obtain a upper bound for $|{\mathcal I}_{s,t,n}^1|$ that
\begin{eqnarray}\label{4.7}
\begin{split}&\int_{s_3}^{s_4}\|\nabla {\mathcal P}_{r-s_2}b(r)\|_0{\mathbb E}^{s_1}|\varphi_r-\varphi_r^n-{\mathbb E}^{s_1}(\varphi_r-\varphi_r^n)|dr+\frac{1}{2}\int_{s_3}^{s_4}\|\nabla^2 {\mathcal P}_{r-s_2}b(r)\|_0
\\
 &\quad \times |{\mathbb E}^{s_1}(\varphi_r-\varphi_r^n)| \big[{\mathbb E}^{s_1}|\varphi_r-{\mathbb E}^{s_1}
\varphi_r|+{\mathbb E}^{s_1}|\varphi_r^n-{\mathbb E}^{s_1}\varphi_r^n|\big]dr.
\end{split}
\end{eqnarray}
This, together with (\ref{1.7}) and (\ref{2.11}) ($g_t=\varphi_t$ and $\varphi_t^n$), implies
\begin{eqnarray}\label{4.8}
\begin{split}|{\mathcal I}_{s,t,n}^1| \leq&C\int_{s_3}^{s_4}[b(r)]_\alpha (r-s_2)^{\frac{\alpha-1}{2}}{\mathbb E}^{s_1}|\varphi_r-\varphi_r^n-{\mathbb E}^{s_1}(\varphi_r-\varphi_r^n)|dr\\ &+C\int_{s_3}^{s_4}[b(r)]_\alpha (r-s_2)^{\frac{\alpha}{2}-1}|{\mathbb E}^{s_1}(\varphi_r-\varphi_r^n)|w_{b,\alpha,q}(s_1,r)^{\frac{1}{q}}(r-s_1)^{1+\frac{\alpha}{2}-\frac{1}{q}}dr\\
\leq&C\int_{s_3}^{s_4}[b(r)]_\alpha(r-s_2)^{\frac{\alpha-1}{2}}{\mathbb E}^{s_1}|\varphi_r-\varphi_r^n-\varphi_{s_1}+\varphi_{s_1}^n|dr\\ &+Cw_{b,\alpha,q}(s_1,s_4)^{\frac{1}{q}}\int_{s_3}^{s_4}[b(r)]_\alpha (r-s_2)^{\alpha-\frac{1}{q}}|{\mathbb E}^{s_1}(\varphi_r-\varphi_r^n)|dr,
\end{split}
\end{eqnarray}
where in the last inequality we used $r-s_2\leq r-s_1\leq 2(r-s_2)$ for $r\in [s_3,s_4]$ and (\ref{1.9}).

For every $p\geq 2$, it follows from  (\ref{1.8}), (\ref{1.1}) and (\ref{1.3}) that
\begin{eqnarray}\label{4.9}
\begin{split} \|{\mathbb E}^{s_1}|\varphi_r-\varphi_r^n-\varphi_{s_1}+\varphi_{s_1}^n|
\|_{L^p(\Omega)}\leq& \|\varphi_r-\varphi_r^n-\varphi_{s_1}+\varphi_{s_1}^n\|_{L^p(\Omega)} \\
\leq& \big[ \|\varphi_r-\varphi_{s_1}\|_{L^p(\Omega)}+\|\varphi_r^n-\varphi_{s_1}^n\|_{L^p(\Omega)}\big]
 \\
\leq&2w_{b,\alpha,q}(s_1,r)^{\frac{1}{q}}(r-s_1)^{1-\frac{1}{q}}.
\end{split}
\end{eqnarray}
Integrating (\ref{4.8}) and (\ref{4.9}), we derive
\begin{eqnarray}\label{4.10}
\begin{split} \|{\mathcal I}_{s,t,n}^1\|_{L^p(\Omega)}
\leq&Cw_{b,\alpha,q}(s_1,s_4)^{\frac{1}{q}}[\varphi-\varphi^n]_{{\mathcal C}_{[s_1,s_4],p}^{w_{b,\alpha,q}^{1/q}}}\int_{s_3}^{s_4}[b(r)]_\alpha (r-s_2)^{\frac{\alpha-1}{2}}dr\\
 &+Cw_{b,\alpha,q}(s_1,s_4)^{\frac{1}{q}}\|\varphi-\varphi^n\|_{[s_3,s_4],p}
 \int_{s_3}^{s_4}[b(r)]_\alpha (r-s_2)^{\alpha-\frac{1}{q}}dr \\
\leq&Cw_{b,\alpha,q}(s_1,s_4)^{\frac{2}{q}}
\Big[[\varphi-\varphi^n]_{{\mathcal C}_{[s_1,s_4],p}^{w_{b,\alpha,q}^{1/q}}}+\|\varphi-\varphi^n\|_{[s_3,s_4],p}\Big]
(s_4-s_2)^{\frac{1+\alpha}{2}-\frac{1}{q}},
\end{split}
\end{eqnarray}
where
\begin{eqnarray}\label{4.11}
[\varphi-\varphi^n]_{{\mathcal C}_{[s_1,s_4],p}^{w_{b,\alpha,q}^{1/q}}}=\sup_{s_1\leq \tau<r\leq s_4}\frac{\|\varphi_r-\varphi_r^n-\varphi_{\tau}+\varphi_{\tau}^n\|_{L^p(\Omega)}}{
w_{b,\alpha,q}(\tau,r)^{\frac{1}{q}}}.
\end{eqnarray}

If we replace $s_2$ with $s_3$, $s_3$ with $s_4$ and $s_4$ with $s_5$, by exactly the same argument we obtain
\begin{eqnarray}\label{4.12}
\|{\mathcal I}_{s,t,n}^2\|_{L^p(\Omega)}
\leq Cw_{b,\alpha,q}(s_1,s_5)^{\frac{2}{q}}\Big[[\varphi-\varphi^n]_{{\mathcal C}_{[s_1,s_5],p}^{w_{b,\alpha,q}^{1/q}}}+\|\varphi-\varphi^n\|_{[s_4,s_5],p}\Big]
(s_5-s_3)^{\frac{1+\alpha}{2}-\frac{1}{q}}.
\end{eqnarray}
Combining the estimate in (\ref{4.12}) with (\ref{4.10}) and (\ref{4.5}), and noting that $w_{b,\alpha,q}$ is a control, we obtain
\begin{eqnarray}\label{4.13}
\|{\mathbb E}^{s_-}\delta A_{s,u,t}\|_{L^p(\Omega)}
\leq Cw_{b,\alpha,q}(s_-,t)^{\frac{2}{q}}(t-s)^{\frac{1+\alpha}{2}-\frac{1}{q}}\Big[[\varphi-\varphi^n]_{{\mathcal C}_{[S_0,T_0],p}^{w_{b,\alpha,q}^{1/q}}}+
 \|\varphi-\varphi^n\|_{[S_0,T_0],p}\Big].
\end{eqnarray}

Now, let us check (\ref{2.9}). For $(s,t)\in [S_0,T_0]_\leq^2$ and $k\in {\mathbb N}$, we denote $t_i:=s+i(t-s)/k$, $i=0,1,\ldots,k$ and set
\begin{eqnarray*}
 {\mathcal I}_{s,t,n}^{(k)}:={\mathcal I}_{t,n}-{\mathcal I}_{s,n}-\sum_{i=1}^{k-1}A_{t_i,t_{i+1}}.
\end{eqnarray*}
Since $t_i-(t_{i+1}-t_i)=t_{i-1}$, $|{\mathcal I}_{s,t,n}^{(k)}|$ can be bounded above by
\begin{eqnarray}\label{4.14}
\begin{split} &\sum_{i=1}^{k-1}\int_{t_i}^{t_{i+1}} |b(r,B_r+\varphi_r)-{\mathbb E}^{t_{i-1}} b(r,B_r+{\mathbb E}^{t_{i-1}}\varphi_r)|dr
\\ &+
\sum_{i=1}^{k-1}\int_{t_i}^{t_{i+1}} |b(r,B_r+\varphi_r^n)-{\mathbb E}^{t_{i-1}}b(r,B_r+{\mathbb E}^{t_{i-1}}\varphi_r^n)|dr\\ &+\int_{t_0}^{t_1} |b(r,B_r+\varphi_r)-b(r,B_r+\varphi_r^n)|dr.
\end{split}
\end{eqnarray}
For the integrand in the first and second integrals in (\ref{4.14}), we have
\begin{eqnarray}\label{4.15}
\begin{split} &|b(r,B_r+\varphi_r)-{\mathbb E}^{t_{i-1}} b(r,B_r+{\mathbb E}^{t_{i-1}}\varphi_r)|\\ \leq&|b(r,B_r+\varphi_r)-b(r,B_{t_{i-1}}+{\mathbb E}^{t_{i-1}}\varphi_r)|\\  &+|{\mathbb E}^{t_{i-1}}[ b(r,B_{t_{i-1}}+{\mathbb E}^{t_{i-1}}\varphi_r)-b(r,B_r+{\mathbb E}^{t_{i-1}}\varphi_r)]|
\\ \leq&
2[b(r)]_\alpha\big[|B_r-B_{t_{i-1}}|^\alpha+|\varphi_r-{\mathbb E}^{t_{i-1}}\varphi_r|^\alpha\big]
\end{split}
\end{eqnarray}
and
\begin{eqnarray}\label{4.16}
\begin{split}
& |b(r,B_r+\varphi_r^n)-{\mathbb E}^{t_{i-1}} b(r,B_r+{\mathbb E}^{t_{i-1}}\varphi_r^n)|
\\ \leq&|b(r,B_r+\varphi_r^n)-b(r,B_{t_{i-1}}+{\mathbb E}^{t_{i-1}}\varphi_r^n)|\\  &+|{\mathbb E}^{t_{i-1}}[ b(r,B_{t_{i-1}}+{\mathbb E}^{t_{i-1}}\varphi_r^n)-b(r,B_r+{\mathbb E}^{t_{i-1}}\varphi_r^n)]|
\\ \leq&
2[b(r)]_\alpha\big[|B_r-B_{t_{i-1}}|^\alpha+|\varphi_r^n-{\mathbb E}^{t_{i-1}}\varphi_r^n|^\alpha\big].
\end{split}
\end{eqnarray}
Suming over (\ref{4.14})--(\ref{4.16}), it follows that
\begin{eqnarray*}
\begin{split} |{\mathcal I}_{s,t,n}^{(k)}|
\leq&4\sum_{i=1}^{k-1}\int_{t_i}^{t_{i+1}} [b(r)]_\alpha [|B_r-B_{t_{i-1}}|^\alpha+|\varphi_r-{\mathbb E}^{t_{i-1}}\varphi_r|^\alpha\\& +|\varphi_r^n-{\mathbb E}^{t_{i-1}}\varphi_r^n|^\alpha\big] dr+2\int_{t_0}^{t_1} \|b(r)\|_0dr.
\end{split}
\end{eqnarray*}
Whence,
\begin{eqnarray}\label{4.17}
\begin{split} \|{\mathcal I}_{s,t,n}^{(k)}\|_{L^1(\Omega)}
\leq& C(d,\alpha)\sum_{i=1}^{k-1}\int_{t_i}^{t_{i+1}} [b(r)]_\alpha \Big\{|r-t_{i-1}|^{\frac{\alpha}{2}}+\big[{\mathbb E}|\varphi_r-{\mathbb E}^{t_{i-1}}\varphi_r|\big]^\alpha\\
 &+
\big[{\mathbb E}|\varphi_r^n-{\mathbb E}^{t_{i-1}}\varphi_r^n|\big]^\alpha\Big\} dr+2\|b\|_{q,\alpha}(t_1-t_0)^{1-\frac{1}{q}}.
\end{split}
\end{eqnarray}
With the help of (\ref{1.8}) and (\ref{2.11}) ($g_t=\varphi_t$), one gets
\begin{eqnarray}\label{4.18}
\begin{split} {\mathbb E}|\varphi_r-{\mathbb E}^{t_{i-1}}\varphi_r|=
{\mathbb E}{\mathbb E}^{t_{i-1}}|\varphi_r-{\mathbb E}^{t_{i-1}}
\varphi_r|
 \leq& Cw_{b,\alpha,q}(t_{i-1},r)^{\frac{1}{q}}(r-t_{i-1})^{1+\frac{\alpha}{2}-\frac{1}{q}}
 \\
 \leq& Cw_{b,\alpha,q}(s,t)^{\frac{1}{q}}(r-t_{i-1})^{1+\frac{\alpha}{2}-\frac{1}{q}}
\end{split}
\end{eqnarray}
and
\begin{eqnarray}\label{4.19}
{\mathbb E}|\varphi_r^n-{\mathbb E}^{t_{i-1}}\varphi_r^n|\leq Cw_{b,\alpha,q}(s,t)^{\frac{1}{q}}(r-t_{i-1})^{1+\frac{\alpha}{2}-\frac{1}{q}}.
\end{eqnarray}
From (\ref{4.18}), (\ref{4.19}), and (\ref{4.17}), we obtain
\begin{eqnarray}\label{4.20}
\begin{split} \|{\mathcal I}_{s,t,n}^{(k)}\|_{L^1(\Omega)}
\leq& C\big[1+w_{b,\alpha,q}(s,t)^{\frac{\alpha}{q}}\big] \sum_{i=1}^{k-1}\int_{t_i}^{t_{i+1}} [b(r)]_\alpha |r-t_{i-1}|^{\frac{\alpha}{2}\wedge (\alpha+\frac{\alpha^2}{2}-\frac{\alpha}{q})} dr +2T^{1-\frac{1}{q}}\|b\|_{q,\alpha}k^{\frac{1}{q}-1}
\\ \leq& C\|b\|_{q,\alpha} \big[k^{-\frac{\alpha}{2}}+k^{\frac{1}{q}-1}\big]\rightarrow 0, \ \ {\rm as} \ \ k\rightarrow\infty.
\end{split}\end{eqnarray}

In view of (\ref{4.4}), (\ref{4.15}), (\ref{4.20}), if one takes the controls $w_1=w_2=w_{b,\alpha,q}^{2/q}$, the constants $\varepsilon_1=\varepsilon_2=(1+\alpha)/2-1/q$ and
$$
\Gamma_1=C\|\varphi-\varphi^n\|_{[S_0,T_0],p}, \ \Gamma_2=C\Big[[\varphi-\varphi^n]_{{\mathcal C}_{[S_0,T_0],p}^{w_{b,\alpha,q}^{1/q}}}+
 \|\varphi-\varphi^n\|_{[S_0,T_0],p}\Big],
$$
in Lemma \ref{lem2.3}, for every $(s,t)\in [S_0,T_0]_\leq^2$ we conclude
\begin{eqnarray}\label{4.21}
\begin{split} \|{\mathcal I}_{t,n}-{\mathcal I}_{s,n}\|_{L^p(\Omega)} \leq& Cw_{b,\alpha,q}(s,t)^{\frac{1}{q}}(t-s)^{\frac{1+\alpha}{2}-\frac{1}{q}}\|\varphi-\varphi^n\|_{[S_0,T_0],p} \\
&+Cw_{b,\alpha,q}(s,t)^{\frac{2}{q}}(t-s)^{\frac{1+\alpha}{2}-\frac{1}{q}}
[\varphi-\varphi^n]_{{\mathcal C}_{[S_0,T_0],p}^{w_{b,\alpha,q}^{1/q}}}\\
\leq& Cw_{b,\alpha,q}(s,t)^{\frac{1}{q}}
(t-s)^{\frac{1+\alpha}{2}-\frac{1}{q}}\Big[\|\varphi_{S_0}-\varphi_{S_0}^n\|_{L^p(\Omega)}+
[\varphi-\varphi^n]_{{\mathcal C}_{[S_0,T_0],p}^{w_{b,\alpha,q}^{1/q}}}\Big],
\end{split} \end{eqnarray}
where the constant $C$ depends only on $d,\alpha,T,p,q$ and $\|b\|_{q,\alpha}$.

\smallskip
\textbf{Step 2:} To estimate ${\mathcal J}_{t,n}$. For $(s,t)\in \overline{[S_0,T_0]}_\leq^2$, set
\begin{eqnarray}\label{4.22}
A_{s,t}={\mathbb E}^{s_-}\int_s^t [b(r,B_r+{\mathbb E}^{s_-}\varphi_r^n)-b(r,B_r+{\mathbb E}^{s_-}\varphi_{\kappa_n(r)}^n)]dr.
\end{eqnarray}
For any $p\in [2,\infty)$, we get
\begin{eqnarray}\label{4.23}
\begin{split} \|A_{s,t}\|_{L^p(\Omega)}\leq& \int_s^t \|{\mathcal P}_{r-s_-}b(r,B_{s_-}+{\mathbb E}^{s_-}\varphi_r^n)-{\mathcal P}_{r-s_-}b(r,B_{s_-}+{\mathbb E}^{s_-}\varphi_{\kappa_n(r)}^n)\|_{L^p(\Omega)}dr\\ \leq& \int_s^t \|\nabla {\mathcal P}_{r-s_-}b(r)\|_0\|\varphi_r^n-\varphi_{\kappa_n(r)}^n\|_{L^p(\Omega)}dr\\ \leq& C(d,\alpha)\int_s^t [b(r)]_\alpha (r-s_-)^{\frac{\alpha-1}{2}}\|\varphi_r^n-\varphi_{\kappa_n(r)}^n\|_{L^p(\Omega)}dr\\ \leq& C(d,\alpha)w_{b,\alpha,q}(0,t)^{\frac{1}{q}}
w_{b,\alpha,q}(s_-,t)^{\frac{1}{q}}(t-s)^{\frac{1+\alpha}{2}-\frac{1}{q}}n^{\frac{1}{q}-1}.
\end{split}
\end{eqnarray}

Let $u$ and $s_i, \ i=1,\ldots,5$ be defined in (\ref{3.16}). Then
\begin{eqnarray}\label{4.24}
{\mathbb E}^{s_-}\delta A_{s,u,t}={\mathbb E}^{s_1}\delta A_{s_3,s_4,s_5}={\mathcal J}_{s,t,n}^1+{\mathcal J}_{s,t,n}^2,
\end{eqnarray}
where
\begin{eqnarray}\label{4.25}
\begin{split}{\mathcal J}_{s,t,n}^1
 =&{\mathbb E}^{s_1}\int_{s_3}^{s_4}[{\mathcal P}_{r-s_2}b(r,B_{s_2}+{\mathbb E}^{s_1}\varphi_r^n)-
 {\mathcal P}_{r-s_2}b(r,B_{s_2}+{\mathbb E}^{s_1}
 \varphi_{\kappa_n(r)}^n)\\
 &\qquad\quad \ \ -{\mathcal P}_{r-s_2}b(r,B_{s_2}+{\mathbb E}^{s_2}
\varphi_r^n)+{\mathcal P}_{r-s_2}b(r,B_{s_2}+{\mathbb E}^{s_2}\varphi_{\kappa_n(r)}^n)]dr
\end{split}
\end{eqnarray}
and
\begin{eqnarray}\label{4.26}
\begin{split}{\mathcal J}_{s,t,n}^2
 =&{\mathbb E}^{s_1}\int_{s_4}^{s_5}[{\mathcal P}_{r-s_3}b(r,B_{s_3}+{\mathbb E}^{s_1}\varphi_r^n)-
 {\mathcal P}_{r-s_3}b(r,B_{s_3}+{\mathbb E}^{s_1}
 \varphi_{\kappa_n(r)}^n)\\
 &\qquad\quad \ \ -{\mathcal P}_{r-s_3}b(r,B_{s_3}+{\mathbb E}^{s_3}
\varphi_r^n)+{\mathcal P}_{r-s_3}b(r,B_{s_3}+{\mathbb E}^{s_3}\varphi_{\kappa_n(r)}^n)]dr.
\end{split}
\end{eqnarray}
For ${\mathcal J}_{s,t,n}^1$, one gets an analogue of (\ref{4.7}) that
\begin{eqnarray}\label{4.27}
\begin{split}|{\mathcal J}_{s,t,n}^1|\leq& C(d,\alpha)\int_{s_3}^{s_4}[b(r)]_\alpha(r-s_2)^{\frac{\alpha-1}{2}}{\mathbb E}^{s_1}|\varphi_r^n-\varphi_{\kappa_n(r)}^n-{\mathbb E}^{s_1}(\varphi_r^n-\varphi_{\kappa_n(r)}^n)|dr\\
 &+C(d,\alpha)\int_{s_3}^{s_4}[b(r)]_\alpha(r-s_2)^{\frac{\alpha-2}{2}}
|{\mathbb E}^{s_1}(\varphi_r^n-\varphi_{\kappa_n(r)}^n)|\\ &\qquad \times \big[{\mathbb E}^{s_1}|\varphi_r^n-{\mathbb E}^{s_1}
\varphi_r^n|+{\mathbb E}^{s_1}|\varphi_{\kappa_n(r)}^n-{\mathbb E}^{s_1}\varphi_{\kappa_n(r)}^n|\big]dr.
\end{split}
\end{eqnarray}
If $\kappa_n(r)\leq s_1<r$, by (\ref{2.11}) ($g_t=\varphi_t^n$) then
\begin{eqnarray}\label{4.28}
\begin{split}
&{\mathbb E}^{s_1}|\varphi_r^n-\varphi_{\kappa_n(r)}^n-{\mathbb E}^{s_1}(\varphi_r^n-\varphi_{\kappa_n(r)}^n)|
 =C{\mathbb E}^{s_1}|\varphi_r^n-{\mathbb E}^{s_1}\varphi_r^n| \\ \leq & Cw_{b,\alpha,q}(s_1,r)^{\frac{1}{q}}(r-s_1)^{1+\frac{\alpha}{2}-\frac{1}{q}}\leq Cw_{b,\alpha,q}(s_1,r)^{\frac{1}{q}}(r-s_1)^{\frac{\alpha}{2}}n^{\frac{1}{q}-1}.
\end{split}
\end{eqnarray}
Otherwise, we use (\ref{1.9}), (\ref{1.3}), (\ref{2.11}) and conditional Jensen's inequality to get
\begin{eqnarray}\label{4.29}
\begin{split}
&{\mathbb E}^{s_1}|\varphi_r^n-\varphi_{\kappa_n(r)}^n-{\mathbb E}^{s_1}(\varphi_r^n-\varphi_{\kappa_n(r)}^n)|\\
 \leq &2{\mathbb E}^{s_1}\Big|\int_{\kappa_n(r)}^r[b(\tau,B_{\kappa_n(r)}+\varphi_{\kappa_n(r)}^n)-
b(\tau,B_{s_1}+{\mathbb E}^{s_1}\varphi_{\kappa_n(r)}^n)]d\tau\Big| \\ \leq & C
\big[(r-s_1)^{\frac{\alpha}{2}}+{\mathbb E}^{s_1}|\varphi_{\kappa_n(r)}^n-{\mathbb E}^{s_1}\varphi_{\kappa_n(r)}^n|^\alpha\big]w_{b,\alpha,q}(s_1,r)^{\frac{1}{q}}n^{\frac{1}{q}-1}\\ \leq & C
\big[(r-s_1)^{\frac{\alpha}{2}}+w_{b,\alpha,q}(s_1,r)^{\frac{\alpha}{q}}
(r-s_1)^{\alpha+\frac{\alpha^2}{2}-\frac{\alpha}{q}}\big]w_{b,\alpha,q}(s_1,r)^{\frac{1}{q}}n^{\frac{1}{q}-1}
\\ \leq & Cw_{b,\alpha,q}(s_1,r)^{\frac{1}{q}}(r-s_1)^{\frac{\alpha}{2}}n^{\frac{1}{q}-1}.
\end{split}
\end{eqnarray}
In view of (\ref{4.27})--(\ref{4.29}) and (\ref{2.11}) ($g_t = \varphi_t^n$), for every $p \in [2,\infty)$ we obtain
\begin{eqnarray}\label{4.30}
\|{\mathcal J}_{s,t,n}^1\|_{L^p(\Omega)}
 \leq  Cw_{b,\alpha,q}(s_1,s_4)^{\frac{2}{q}}(s_4-s_1)^{\alpha+\frac{1}{2}-\frac{1}{q}}
n^{\frac{1}{q}-1}.
\end{eqnarray}
 A similar discussion, adapted to ${\mathcal J}_{s,t,n}^2$, yields
\begin{eqnarray}\label{4.31}
\|{\mathcal J}_{s,t,n}^2\|_{L^p(\Omega)}
 \leq  Cw_{b,\alpha,q}(s_1,s_5)^{\frac{2}{q}}(s_5-s_1)^{\alpha+\frac{1}{2}-\frac{1}{q}}
n^{\frac{1}{q}-1}.
\end{eqnarray}
Repeating the calculations in (\ref{4.16})--(\ref{4.20}) with ${\mathcal J}_{t,n}$ and ${\mathcal J}_{s,n}$ in place of ${\mathcal I}_{t,n}$ and ${\mathcal I}_{s,n}$,  respectively,
yields (\ref{2.9}) for ${\mathcal J}_{t,n}$.
 By (\ref{4.23}), (\ref{4.24}), (\ref{4.30}), (\ref{4.31}) and Lemma \ref{lem2.3}, there exists a constant $C=C(d,\alpha,T,p,q,\|b\|_{q,\alpha})$ such that
\begin{eqnarray}\label{4.32}
\|{\mathcal J}_{t,n}-{\mathcal J}_{s,n}\|_{L^p(\Omega)}
\leq  C
w_{b,\alpha,q}(s,t)^{\frac{1}{q}}(t-s)^{\frac{1+\alpha}{2}-\frac{1}{q}}n^{\frac{1}{q}-1}.
\end{eqnarray}

\textbf{Step 3:} To estimate ${\mathcal K}_{t,n}$. In this step we always assume $\varepsilon_1\in [0,(1+\alpha)/2-1/q]$. For $(s,t)\in \overline{[S_0,T_0]}_\leq^2$, set
\begin{eqnarray}\label{4.33}
A_{s,t}={\mathbb E}^{s_-}\int_s^t [b(r,B_r+{\mathbb E}^{s_-}\varphi_{\kappa_n(r)}^n)-b(r,B_{\kappa_n(r)}+{\mathbb E}^{s_-}\varphi_{\kappa_n(r)}^n)]dr.
\end{eqnarray}

If $s\leq t\leq s+2T/n$, then for every $p\geq 2$
\begin{eqnarray}\label{4.34}
\begin{split} \quad \|A_{s,t}\|_{L^p(\Omega)}\leq& \int_s^t \|b(r,B_r+{\mathbb E}^{s_-}\varphi_{\kappa_n(r)}^n)-b(r,B_{\kappa_n(r)}+{\mathbb E}^{s_-}
\varphi_{\kappa_n(r)}^n)\|_{L^p(\Omega)}dr \\ \leq& C w_{b,\alpha,q}(s,t)^{\frac{1}{q}}(t-s)^{1+\frac{\alpha}{2}-\frac{1}{q}} \leq C w_{b,\alpha,q}(s,t)^{\frac{1}{q}}(t-s)^{\varepsilon_1} n^{-1-\frac{\alpha}{2}+\frac{1}{q}+\varepsilon_1}.
\end{split}
\end{eqnarray}

If $t> s+2T/n$, then for $r\geq s$ and $\kappa_n(r)\geq s_-+(t-s)/2$. Thus,
\begin{eqnarray}\label{4.35}
\begin{split} \|A_{s,t}\|_{L^p(\Omega)}=& \Big\|\int_s^t ({\mathcal P}_{r-s_-}-{\mathcal P}_{\kappa_n(r)-s_-})b(r,B_{s_-}+{\mathbb E}^{s_-}\varphi_{\kappa_n(r)}^n)dr\Big\|_{L^p(\Omega)} \\ \leq&  \int_s^t dr\int_{\kappa_n(r)-s_-}^{r-s_-}\|\partial_\tau {\mathcal P}_\tau b(r)\|_0 d\tau\\ \leq& C\int_s^t[b(r)]_\alpha dr \int_{\kappa_n(r)-s_-}^{r-s_-}\tau^{\frac{\alpha}{2}-1}d\tau \\ \leq& C\int_s^t[b(r)]_\alpha (\kappa_n(r)-s_-)^{-1+\frac{1}{q}+\varepsilon_1} dr \int_{\kappa_n(r)-s_-}^{r-s_-}\tau^{\frac{\alpha}{2}-\frac{1}{q}-\varepsilon_1}d\tau \\ \leq & C w_{b,\alpha,q}(s,t)^{\frac{1}{q}}(t-s)^{\varepsilon_1} n^{-1-\frac{\alpha}{2}+\frac{1}{q}+\varepsilon_1},
\end{split}
\end{eqnarray}
where in the last inequality we have used H\"{o}lder's inequality and the following fact
\begin{eqnarray*}
\int_s^t(\kappa_n(r)-s_-)^{-1+\frac{q\varepsilon_1}{q-1}} dr \leq (\kappa_n(s)-s_-)^{-1+\frac{q\varepsilon_1}{q-1}}(t-s)
\leq 2(t-s)^{\frac{q\varepsilon_1}{q-1}}.
\end{eqnarray*}

Let $u$ and $s_i$ $(i=1,\ldots 5)$ be stated in (\ref{4.5}). Then
\begin{eqnarray}\label{4.36}
\begin{split} {\mathbb E}^{s_-}\delta A_{s,u,t} =&{\mathbb E}^{s_1}{\mathbb E}^{s_2}\int_{s_3}^{s_4}[b(r,B_r+{\mathbb E}^{s_1}\varphi_{\kappa_n(r)}^n)-
b(r,B_{\kappa_n(r)}+{\mathbb E}^{s_1}\varphi_{\kappa_n(r)}^n)\\&-b(r,B_r+{\mathbb E}^{s_2}
\varphi_{\kappa_n(r)}^n) +b(r,B_{\kappa_n(r)}+{\mathbb E}^{s_2}\varphi_{\kappa_n(r)}^n)]dr\\ &+
{\mathbb E}^{s_1}{\mathbb E}^{s_3}\int_{s_4}^{s_5}[b(r,B_r+{\mathbb E}^{s_1}\varphi_{\kappa_n(r)}^n)
-b(r,B_{\kappa_n(r)}+{\mathbb E}^{s_1}\varphi_{\kappa_n(r)}^n)
\\& -b(r,B_r+{\mathbb E}^{s_3}
\varphi_{\kappa_n(r)}^n)+b(r,B_{\kappa_n(r)}+{\mathbb E}^{s_3}\varphi_{\kappa_n(r)}^n)]dr  =: {\mathcal K}_{s,t,n}^1+{\mathcal K}_{s,t,n}^2.
\end{split}
\end{eqnarray}
The two terms are treated in exactly the same way, so we only discuss ${\mathcal K}_{s,t,n}^1$.

\smallskip
$\bullet$ \textbf{Case 1: $t-s\geq 4T/n$}. In this case, we have $\kappa_n(r)-s_2\geq s_3-T/n-s_2\geq (t-s)/4>0$ for every $r\in [s_3,s_4]$. We rewrite ${\mathcal K}_{s,t,n}^1$ by
\begin{eqnarray*}
\begin{split} {\mathcal K}_{s,t,n}^1=&{\mathbb E}^{s_1}\int_{s_3}^{s_4}\big[({\mathcal P}_{r-s_2}-{\mathcal P}_{\kappa_n(r)-s_2}) b(r,B_{s_2}+{\mathbb E}^{s_1}\varphi_{\kappa_n(r)}^n)\\
&\qquad\qquad -
({\mathcal P}_{r-s_2}-{\mathcal P}_{\kappa_n(r)-s_2}) b(r,B_{s_2}+{\mathbb E}^{s_2}\varphi_{\kappa_n(r)}^n)\big]dr,
\end{split}
\end{eqnarray*}
then $|{\mathcal K}_{s,t,n}^1|$ is bounded above by
\begin{eqnarray}\label{4.37}
\begin{split} &{\mathbb E}^{s_1}\int_{s_3}^{s_4}\|\nabla({\mathcal P}_{r-s_2}-{\mathcal P}_{\kappa_n(r)-s_2}) b(r)\|_0|{\mathbb E}^{s_1}\varphi_{\kappa_n(r)}^n -{\mathbb E}^{s_2}\varphi_{\kappa_n(r)}^n|dr
\\ \leq &
\int_{s_3}^{s_4}{\mathbb E}^{s_1}|\varphi_{\kappa_n(r)}^n -{\mathbb E}^{s_1}\varphi_{\kappa_n(r)}^n)|dr\int_{\kappa_n(r)-s_2}^{r-s_2} \|\partial\tau \nabla{\mathcal P}_\tau b(r)\|_0d\tau
\\ \leq & C\int_{s_3}^{s_4}[b(r)]_\alpha w_{b,\alpha,q}(s_1,\kappa_n(r))^{\frac{1}{q}}(\kappa_n(r)-s_1)^{1+\frac{\alpha}{2}-\frac{1}{q}}
dr\int_{\kappa_n(r)-s_2}^{r-s_2} \tau^{\frac{\alpha-3}{2}}d\tau \\ \leq & Cw_{b,\alpha,q}(s_1,s_4)^{\frac{1}{q}}\int_{s_3}^{s_4}[b(r)]_\alpha
(\kappa_n(r)-s_1)^{\frac{\alpha-1}{2}}
dr \int_{\kappa_n(r)-s_2}^{r-s_2}\tau^{\frac{\alpha}{2}-\frac{1}{q}}d\tau,
\end{split}
\end{eqnarray}
where in the second inequality we have used (\ref{2.11}) ($g_t=\varphi^n_t$) and in the last inequality we used
$$
\kappa_n(r)-s_2\leq \kappa_n(r)-s_1\leq 2(\kappa_n(r)-s_2), \quad \forall \ r\in [s_3,s_4].
$$

By (\ref{4.37}), for every $p\geq 2$, we obtain
\begin{eqnarray}\label{4.38}
\|{\mathcal K}_{s,t,n}^1\|_{L^p(\Omega)} \leq Cw_{b,\alpha,q}(s_1,s_4)^{\frac{2}{q}}(t-s)^{\frac{1+\alpha}{2}-\frac{1}{q}} n^{-1-\frac{\alpha}{2}+\frac{1}{q}}.
\end{eqnarray}
The bounds in (\ref{4.38}) holds true for ${\mathcal K}_{s,t,n}^2$ as well. Thus, we conclude
\begin{eqnarray}\label{4.39}
\|{\mathbb E}^{s_-}\delta A_{s,u,t}\|_{L^p(\Omega)} \leq  Cw_{b,\alpha,q}(s_-,t)^{\frac{2}{q}} (t-s)^{\frac{1+\alpha}{2}-\frac{1}{q}} n^{-1-\frac{\alpha}{2}+\frac{1}{q}}.
\end{eqnarray}

$\bullet$ \textbf{Case 2: $t-s< 4T/n$}. Noting that $\varphi_{\kappa_n(r)}^n$ is ${\mathcal F}_{(\kappa_n(r)-\frac{T}{n})\vee 0}$ measurable. Therefore, if $r\in [\kappa_n(s_1),\kappa_n(s_1)+2T/n)$, then $\varphi_{\kappa_n(r)}^n$ is ${\mathcal F}_{\kappa_n(s_1)}$ measurable. Since $\kappa_n(s_1)\leq s_1\leq s_2$, one has ${\mathbb E}^{s_1}\varphi_{\kappa_n(r)}^n={\mathbb E}^{s_2}\varphi_{\kappa_n(r)}^n=\varphi_{\kappa_n(r)}^n$ and thus the integrand in ${\mathcal K}_{s,t,n}^1$ is zero. Hence, we only concentrate on the case of $r\geq \kappa_n(s_1)+2T/n$. Now $\kappa_n(r)> \kappa_n(s_1)+T/n\geq s_1$, so ${\mathcal K}_{s,t,n}^1$ can be rewrote by
\begin{eqnarray*}
\begin{split} {\mathcal K}_{s,t,n}^1 =&{\mathbb E}^{s_1}\int_{[s_3,s_4]\cap [\kappa_n(s_1)+2T/n,s_4]}\big[{\mathcal P}_{\frac{T}{n}}b(r,B_{\kappa_n(r)-\frac{T}{n}}+{\mathbb E}^{s_2}\varphi_{\kappa_n(r)}^n)
\\ & \qquad -{\mathcal P}_{\frac{T}{n}}b(r,B_{\kappa_n(r)-\frac{T}{n}}+{\mathbb E}^{s_1}\varphi_{\kappa_n(r)}^n)\big] +\big[{\mathcal P}_{r-\kappa_n(r)+\frac{T}{n}}b(r,B_{\kappa_n(r)-\frac{T}{n}}+{\mathbb E}^{s_1}\varphi_{\kappa_n(r)}^n)
\\  & \qquad -{\mathcal P}_{r-\kappa_n(r)+\frac{T}{n}}b(r,B_{\kappa_n(r)-\frac{T}{n}}+{\mathbb E}^{s_2}\varphi_{\kappa_n(r)}^n)\big]dr.
\end{split}
\end{eqnarray*}
Thus
\begin{eqnarray}\label{4.40}
\begin{split} |{\mathcal K}_{s,t,n}^1|\leq &\int_{[s_3,s_4]\cap [\kappa_n(s_1)+2T/n,s_4]}\big[\|\nabla{\mathcal P}_{\frac{T}{n}}b(r)\|_0+\|\nabla{\mathcal P}_{r-\kappa_n(r)
+\frac{T}{n}}b(r)\|_0\big]\\ &\times {\mathbb E}^{s_1}|{\mathbb E}^{s_2}\varphi_{\kappa_n(r)}^n
-{\mathbb E}^{s_1}\varphi_{\kappa_n(r)}^n|dr.
\end{split}
\end{eqnarray}
By (\ref{1.7}), (\ref{2.11}) ($g_t=\varphi^n_t$) and (\ref{4.40}), it follows that
\begin{eqnarray}\label{4.41}
\begin{split} |{\mathcal K}_{s,t,n}^1|\leq &C\int_{[s_3,s_4]\cap [\kappa_n(s_1)+2T/n,s_4]}[b(r)]_\alpha n^{\frac{1-\alpha}{2}}w_{b,\alpha,q}(s_1,\kappa_n(r))^{\frac{1}{q}}(\kappa_n(r)-s_1)^{1+\frac{\alpha}{2}-\frac{1}{q}}
dr
\\ \leq &Cw_{b,\alpha,q}(s_-,t)^{\frac{2}{q}}n^{\frac{1-\alpha}{2}}\Big[\int_{s_3}^{s_4} (r-s_1)^{\frac{\alpha q}{2(q-1)}+1}dr\Big]^{1-\frac{1}{q}}\\ \leq &  Cw_{b,\alpha,q}(s_-,t)^{\frac{2}{q}} (t-s)^{\frac{1+\alpha}{2}-\frac{1}{q}} n^{-1-\frac{\alpha}{2}+\frac{1}{q}}.
\end{split}
\end{eqnarray}
Therefore, (\ref{4.39}) holds mutatis mutandis for $(s,t)\in \overline{[S_0,T_0]}_\leq^2$ and $u=(t+s)/2$.

It remains to verify that the process ${\mathcal K}_{t,n}$ satisfies (\ref{2.9}). Let $t_i$ and $k$ be stated in (\ref{4.16}). Repeating the calculations from  (\ref{4.16}) to (\ref{4.19}), we get an analogue of (\ref{4.20}) that
\begin{eqnarray}\label{4.42}
\begin{split} &\Big\|{\mathcal K}_{t,n}-{\mathcal K}_{s,n}-\sum_{i=1}^{k-1}A_{t_i,t_{i+1}}\Big\|_{L^1(\Omega)}
\\ \leq& C\big[1+w_{b,\alpha,q}(s,t)^{\frac{\alpha}{q}}\big] \sum_{i=1}^{k-1}\int_{t_i}^{t_{i+1}} [b(r)]_\alpha |r-t_{i-1}|^{\frac{\alpha}{2}\wedge (\frac{\alpha^2}{2}+\alpha-\frac{\alpha}{q})} dr \\ &+2T^{1-\frac{1}{q}}\|b\|_{q,\alpha}k^{-1-\frac{1}{q}}
\\ \leq& C\big[1+w_{b,\alpha,q}(s,t)^{\frac{\alpha}{q}}\big] \|b\|_{q,\alpha} \big[k^{-\frac{\alpha}{2}}+k^{\frac{1}{q}-1}\big]\rightarrow 0, \ \ {\rm as} \ \ k\rightarrow\infty.
\end{split}
\end{eqnarray}
We take controls $w_1=w_2=w_{b,\alpha,q}^{2/q}$, constants $\varepsilon_2=(1+\alpha)/2-1/q$,
$\Gamma_1=Cn^{-1-\alpha/2+1/q+\varepsilon_1}$ and $\Gamma_2=Cn^{-1-\alpha/2+1/q}$
in Lemma \ref{lem2.3}. By (\ref{4.35}), (\ref{4.39}) and (\ref{4.42}), then
\begin{eqnarray}\label{4.43}
\begin{split} \|{\mathcal K}_{t,n}-{\mathcal K}_{s,n}\|_{L^p(\Omega)}  \leq& C w_{b,\alpha,q}(s,t)^{\frac{1}{q}}(t-s)^{\varepsilon_1} n^{-1-\frac{\alpha}{2}+\frac{1}{q}+\varepsilon_1}\\ &+ Cw_{b,\alpha,q}(s,t)^{\frac{2}{q}} (t-s)^{\frac{1+\alpha}{2}-\frac{1}{q}} n^{-1-\frac{\alpha}{2}+\frac{1}{q}} \\ \leq& C w_{b,\alpha,q}(s,t)^{\frac{1}{q}}(t-s)^{\varepsilon_1} n^{-1-\frac{\alpha}{2}+\frac{1}{q}+\varepsilon_1}.
\end{split}
\end{eqnarray}

\textbf{Step 4:} To estimate $X_t-X_t^n$. By (\ref{4.21}), (\ref{4.32}) and (\ref{4.43}), there exists a positive $C$ which depends on $d,T,\alpha$ and $\|b\|_{q,\alpha}$ such that for every $(s,t)\in [S_0,T_0]_\leq^2$ and $p\in [2,\infty)$ we have
\begin{eqnarray*}
\begin{split} &\|X_t-X_t^n-X_s+X_s^n\|_{L^p(\Omega)} =\|\varphi_t-\varphi_t^n-(\varphi_s-\varphi_s^n)\|_{L^p(\Omega)}
\\ \leq& Cw_{b,\alpha,q}(s,t)^{\frac{1}{q}}
(t-s)^{\frac{1+\alpha}{2}-\frac{1}{q}}\Big[\|\varphi_{S_0}-\varphi_{S_0}^n\|_{L^p(\Omega)}+
[\varphi-\varphi^n]_{{\mathcal C}_{[S_0,T_0],p}^{w_{b,\alpha,q}^{1/q}}}\Big] \\
&+C w_{b,\alpha,q}(s,t)^{\frac{1}{q}}
\Big[(t-s)^{\frac{1+\alpha}{2}-\frac{1}{q}}n^{\frac{1}{q}-1}+
(t-s)^{\varepsilon_1} n^{-1-\frac{\alpha}{2}+\frac{1}{q}+\varepsilon_1}\Big].
\end{split}
\end{eqnarray*}
If one takes $\varepsilon_1=(1+\alpha)/2-1/q$, then
\begin{eqnarray}\label{4.44}
\begin{split} &\|X_t-X_t^n-X_s+X_s^n\|_{L^p(\Omega)} \\ \leq& Cw_{b,\alpha,q}(s,t)^{\frac{1}{q}}
(t-s)^{\frac{1+\alpha}{2}-\frac{1}{q}}\Big[\|\varphi_{S_0}-\varphi_{S_0}^n\|_{L^p(\Omega)}+
[\varphi-\varphi^n]_{{\mathcal C}_{[S_0,T_0],p}^{w_{b,\alpha,q}^{1/q}}}+n^{\frac{1}{q}-1}\Big].
\end{split}
\end{eqnarray}

Let $C=C(d,\alpha,T,p,q,\|b\|_{q,\alpha})$ be given in (\ref{4.44}). There exists a positive integer $N\in \mathbb{N}$ such that $CT^{(1+\alpha)/2-1/q}N^{-(1+\alpha)/2+1/q}\leq 1/2$. Set $t_i=iT/N$, $i=0,1,\ldots,N$. Firstly, we assume $S_0=t_0$ and $T_0=t_1$. By (\ref{4.44}), for every $(s,t)\in [t_0,t_1]_\leq^2$ we obtain
\begin{eqnarray*}
\frac{\|\varphi_t-\varphi_t^n-\varphi_s+\varphi_s^n\|_{L^p(\Omega)}}{w_{b,\alpha,q}(s,t)^{\frac{1}{q}}} \leq C(t-s)^{\frac{1+\alpha}{2}-\frac{1}{q}}\Big[
\|\varphi-\varphi^n\|_{{\mathcal C}_{[t_0,t_1],p}^{w_{b,\alpha,q}^{1/q}}}
+ n^{\frac{1}{q}-1}\Big].
 \end{eqnarray*}
Thus,
\begin{eqnarray*}
[\varphi-\varphi^n]_{{\mathcal C}_{[t_0,t_1],p}^{w_{b,\alpha,q}^{1/q}}} \leq \frac{1}{2}[\varphi-\varphi^n]_{{\mathcal C}_{[t_0,t_1],p}^{q,w_{b,\alpha,q}^\prime}}
+\frac{1}{2}n^{\frac{1}{q}-1}.
 \end{eqnarray*}
This, together with (\ref{4.44}), leads to
\begin{eqnarray}\label{4.45}
\|X_t-X_t^n-X_s+X_s^n\|_{L^p(\Omega)}  \leq 2Cw_{b,\alpha,q}(s,t)^{\frac{1}{q}}(t-s)^{\frac{1+\alpha}{2}-\frac{1}{q}} n^{\frac{1}{q}-1}, \ \forall \ p\in [2,\infty).
 \end{eqnarray}

We choose $p$ large enough such that $p(1+\alpha)/2-p/q>1$. By (\ref{4.45}) and Lemma \ref{lem2.5}, $\{X_t-X_t^n\}_{t\in [0,T]}$ has a continuous realization (still denoted by itself). Moreover, by (\ref{2.16}), (\ref{2.17}) and H\"{o}lder's inequality, for $\gamma\in (0,(1+\alpha)/2-1/q)$ and $p\in [2,\infty)$, there is a constant $C=C(d,\alpha,T,p,q,\gamma,\|b\|_{q,\alpha})$ such that
\begin{eqnarray}\label{4.46}
{\mathbb E}\sup_{t_0\leq s<t\leq t_1}\frac{|X_t-X_t^n-X_s+X_s^n|^p}{(t-s)^{p\gamma}w_{b,\alpha,q}(s,t)^{\frac{p}{q}}}\leq Cn^{-p(\frac{1}{q}-1)}.
\end{eqnarray}
Next, we choose $S_0=t_1$, $T_0=t_2$ and by (\ref{4.44})-(\ref{4.45}) to get
\begin{eqnarray*}
\begin{split} &\|\varphi_t-\varphi_t^n-(\varphi_s-\varphi_s^n)\|_{L^p(\Omega)}
\\ \leq& Cw_{b,\alpha,q}(s,t)^{\frac{1}{q}}
(t-s)^{\frac{1+\alpha}{2}-\frac{1}{q}}\Big[\|\varphi_{t_1}-\varphi_{t_1}^n\|_{L^p(\Omega)}+
[\varphi-\varphi^n]_{{\mathcal C}_{[t_1,t_2],p}^{w_{b,\alpha,q}^{1/q}}}+n^{\frac{1}{q}-1}\Big]
\\ \leq& Cw_{b,\alpha,q}(s,t)^{\frac{1}{q}}
(t-s)^{\frac{1+\alpha}{2}-\frac{1}{q}}\Big[w_{b,\alpha,q}(t_0,t_1)^{\frac{1}{q}}n^{\frac{1}{q}-1}+
[\varphi-\varphi^n]_{{\mathcal C}_{[t_1,t_2],p}^{w_{b,\alpha,q}^{1/q}}}+n^{\frac{1}{q}-1}\Big],
\end{split}
\end{eqnarray*}
which also yields
\begin{eqnarray*}
[\varphi-\varphi^n]_{{\mathcal C}_{[t_1,t_2],p}^{w_{b,\alpha,q}^{1/q}}} \leq  \big[1+w_{b,\alpha,q}(t_0,t_1)^{\frac{1}{q}}\big]n^{\frac{1}{q}-1}.
 \end{eqnarray*}
Therefore,
\begin{eqnarray}\label{4.47}
{\mathbb E}\sup_{t_1\leq s<t\leq t_2}\frac{|X_t-X_t^n-X_s+X_s^n|^p}{(t-s)^{p\gamma}w_{b,\alpha,q}(s,t)^{\frac{p}{q}}}\leq Cn^{-p(\frac{1}{q}-1)}.
\end{eqnarray}
The combination of (\ref{4.46}) and (\ref{4.47}) yields
\begin{eqnarray}\label{4.48}
\|X-X^n\|_{p,\gamma,q,w_{b,\alpha,q},[t_0,t_2]}\leq  C(d,\alpha,T,p,q,\gamma,\|b\|_{q,\alpha})
n^{\frac{1}{q}-1}.
\end{eqnarray}
We then repeat the proceeding arguments to extend estimate (\ref{4.48}) to the time interval
$[t_2,t_3]$. Continuing this process inductively, we guarantee (\ref{1.17}).

\section{Proof of Theorem \ref{the1.5}}\label{sec5}\setcounter{equation}{0}
Let $N_0\in \mathbb{N}$  be specified later and set $t_i=iT/N_0,i=0,1,\ldots,N_0$. Let $\hat{X}^{(n)}$ be given by (\ref{1.18}). Observing that
\begin{eqnarray*}
\|X-\hat{X}^{(n)}\|_{p,\gamma,q,w_{b,\alpha,q},T}\leq \sum_{i=0}^{N_0-1}\|X-\hat{X}^{(n)}\|_{p,\gamma,q,w_{b,\alpha,q},[t_i,t_{i+1}]},
\end{eqnarray*}
it is sufficient to show the convergence on each interval $[t_i,t_{i+1}], i=0,1,\ldots,N_0-1$. For notational convenience, define $\gamma_0:=(1+\alpha)/2-1/q$. For $n>m\geq 0$, define the difference process $\hat{X}^{(m,n)}:=\hat{X}^{(n)}-\hat{X}^{(m)}$, and for $(s,t)\in [0,T]_\leq^2$, set
$$
\hat{X}^{(m,n)}_t:=\hat{X}^{(n)}_t-\hat{X}^{(m)}_t \quad {\rm and}\quad
\hat{X}^{(m,n)}_{s,t}:=\hat{X}^{(m,n)}_t-\hat{X}^{(m,n)}_s.
$$
Let  $\hat{\varphi}^{(n)}_t=\hat{X}^{(n)}_t-B_t$. Similarly, we define $\hat{\varphi}^{(m,n)}:=\hat{\varphi}^{(m)}-\hat{\varphi}^{(n)}$, $\hat{\varphi}^{(m,n)}_t:=\hat{\varphi}^{(m)}_t-\hat{\varphi}^{(n)}_t$ and $\hat{\varphi}^{(m,n)}_{s,t}:=\hat{\varphi}^{(m)}_{s,t}-\hat{\varphi}^{(n)}_{s,t}$

By (\ref{1.18}), for every $n\geq 2$,
\begin{eqnarray}\label{5.1}
\begin{split} \hat{X}^{(n-1,n)}_t=&\int_0^t[b(r,\hat{X}^{(n-1)}_r)-b(r,\hat{X}^{(n-2)}_r)]dr
\\ =& \int_0^t[b(r,B_r+\hat{\varphi}^{(n-1)}_r)-b(r,B_r+\hat{\varphi}^{(n-2)}_r)]dr, \quad t\in (0,T].
\end{split}
\end{eqnarray}
We use $\hat{\varphi}^{(n-1)}$ and $\hat{\varphi}^{(n-2)}$ instead of $\varphi$ and $\varphi^n$, respectively, in $\mathcal{I}_{t,n}$ in (\ref{4.1}) and set
\begin{eqnarray}\label{5.2}
A_{s,t}^{(n)}={\mathbb E}^{s_-}\int_s^t [b(r,B_r+{\mathbb E}^{s_-}\hat{\varphi}^{(n-1)}_r)-b(r,B_r+{\mathbb E}^{s_-}\hat{\varphi}^{(n-2)}_r)]dr.
\end{eqnarray}
Observe that (\ref{2.11}) ($g_t=\hat{\varphi}^{(n)}_t$) holds true for all $n\geq 0$ ($n=0$ is natural since $\hat{\varphi}^{(0)}\equiv x$). Let $s_i$, $i=1,\ldots,5$ be given in (\ref{3.16}) and define
$$
\delta A_{s_3,s_4,s_5}^{(n)}=A_{s_3,s_5}^{(n)}-A_{s_3,s_4}^{(n)}-A_{s_4,s_5}^{(n)}.
$$
Repeating calculations from (\ref{4.3}) to (\ref{4.10}), for $(S_0,T_0)\in [0,T]_\leq^2$, $(s,t)\in \overline{[S_0,T_0]}_\leq^2$, $u=(t+s)/2$ and $p\in [2,\infty)$, it follows that
\begin{eqnarray}\label{5.3}
\begin{split}\|{\mathbb E}^{s_1}\delta A_{s_3,s_4,s_5}^{(n)}\|_{L^p(\Omega)}
\leq&C\int_{s_3}^{s_4}[b(r)]_\alpha(r-s_2)^{\frac{\alpha-1}{2}}
\|\hat{\varphi}^{(n-2,n-1)}_{s_1,r}\|_{L^p(\Omega)}dr+
C\int_{s_4}^{s_5}[b(r)]_\alpha(r-s_3)^{\frac{\alpha-1}{2}}
\\ & \times \|\hat{\varphi}^{(n-2,n-1)}_{s_1,r}\|_{L^p(\Omega)}dr+Cw_{b,\alpha,q}(s_1,s_5)^{\frac{1}{q}}\|\hat{\varphi}^{(n-2,n-1)}\|_{[s_3,s_5],p}
\\ & \times \Big[\int_{s_3}^{s_4}[b(r)]_\alpha (r-s_2)^{\alpha-\frac{1}{q}}dr +\int_{s_4}^{s_5}[b(r)]_\alpha (r-s_3)^{\alpha-\frac{1}{q}}dr\Big]
\end{split}
\end{eqnarray}
and
\begin{eqnarray}\label{5.4}
\|A_{s,t}^{(n)}\|_{L^p(\Omega)}\leq Cw_{b,\alpha,q}(s_-,t)^{\frac{1}{q}}(t-s)^{\gamma_0}\|\hat{\varphi}^{(n-2,n-1)}\|_{[s,t],p},
\end{eqnarray}
where the constant $C$ is chosen to be equal to or larger than $1$ and depends only on $d,\alpha,T,p,q$ and $\|b\|_{q,\alpha}$.

By (\ref{5.1}) and the fact $\hat{\varphi}^{(n-2,n-1)}=\hat{X}^{(n-2,n-1)}$,  for every $r\in [s_3,s_5]$, one has
\begin{eqnarray*}
\|\hat{\varphi}^{(n-2,n-1)}_{s_1,r}\|_{L^p(\Omega)}
\leq 2w_{b,\alpha,q}(s_1,r)^{\frac{1}{q}}(r-s_1)^{1-\frac{1}{q}}.
\end{eqnarray*}
Observing that $\gamma_0<1-1/q$, one derives
\begin{eqnarray}\label{5.5}
\|\hat{\varphi}^{(n-2,n-1)}_{s_1,r}\|_{L^p(\Omega)} \leq [\hat{\varphi}^{(n-2,n-1)}]_{\mathcal{C}_{[s_1,s_j],p}^{\gamma_0,w_{b,\alpha,q}^{1/q}}}
w_{b,\alpha,q}(s_1,r)^{\frac{1}{q}}(r-s_1)^{\gamma_0},
\end{eqnarray}
where $\mathcal{C}_{[s_1,s_j],p}^{\gamma_0,w_{b,\alpha,q}^{1/q}}$ is given in (\ref{1.13}) and $j=4$ or $5$.

In view of (\ref{5.3}) and (\ref{5.5}), we get an analogue of (\ref{4.13}) that
\begin{eqnarray}\label{5.6}
 \|{\mathbb E}^{s_-}\delta A_{s,u,t}^{(n)}\|_{L^p(\Omega)}
\leq Cw_{b,\alpha,q}(s_-,t)^{\frac{2}{q}}(t-s)^{2\gamma_0}\Big[
[\hat{\varphi}^{(n-2,n-1)}]_{\mathcal{C}_{[S_0,T_0],p}^{\gamma_0,w_{b,\alpha,q}^{1/q}}}+
 \|\hat{\varphi}^{(n-2,n-1)}\|_{[S_0,T_0],p}\Big].
\end{eqnarray}
The manipulations of estimate (\ref{2.9}) is exactly the same as in (\ref{4.14})--(\ref{4.20}) if one takes $\mathcal{A}_t=\hat{X}^{(n-1,n)}_t$.  Combining (\ref{5.4}), (\ref{5.6}) and (\ref{2.10}), for every $(s,t)\in [S_0,T_0]_\leq^2$, we arrive at
\begin{eqnarray*}
\begin{split} \|\hat{X}^{(n-1,n)}_{s,t}\|_{L^p(\Omega)} \leq & Cw_{b,\alpha,q}(s,t)^{\frac{1}{q}}(t-s)^{\gamma_0}\Big[\|\hat{X}^{(n-2,n-1)}\|_{[S_0,T_0],p}
\\ &+w_{b,\alpha,q}(s,t)^{\frac{1}{q}}(t-s)^{\gamma_0}
[\hat{X}^{(n-2,n-1)}]_{\mathcal{C}_{[S_0,T_0],p}^{\gamma_0,w_{b,\alpha,q}^{1/q}}}\Big],
\end{split} \end{eqnarray*}
which also indicates that there exists another constant $C_0=C_0(d,\alpha,T,p,q,\|b\|_{q,\alpha})\geq 1$ such that
\begin{eqnarray}\label{5.7}
\begin{split} \|\hat{X}^{(n-1,n)}_{s,t}\|_{L^p(\Omega)}
\leq& C_0w_{b,\alpha,q}(s,t)^{\frac{1}{q}}
(t-s)^{\gamma_0}\Big[\|\hat{X}^{(n-2,n-1)}_{S_0}\|_{L^p(\Omega)}\\&+w_{b,\alpha,q}(S_0,T_0)^{\frac{1}{q}}(T_0-S_0)^{\gamma_0}
[\hat{X}^{(n-2,n-1)}]_{\mathcal{C}_{[S_0,T_0],p}^{\gamma_0,w_{b,\alpha,q}^{1/q}}}\Big].
\end{split} \end{eqnarray}
Therefore,
\begin{eqnarray}\label{5.8}
[\hat{X}^{(n-1,n)}]_{\mathcal{C}_{[S_0,T_0],p}^{\gamma_0,w_{b,\alpha,q}^{1/q}}}
 \leq C_0\Big[\|\hat{X}^{(n-2,n-1)}_{S_0}\|_{L^p(\Omega)}+w_{b,\alpha,q}(S_0,T_0)^{\frac{1}{q}}(T_0-S_0)^{\gamma_0}[\hat{X}^{(n-2,n-1)}]_{\mathcal{C}_{[S_0,T_0],p}^{\gamma_0,w_{b,\alpha,q}^{1/q}}}\Big].
\end{eqnarray}

We take $N_0$ large enough such that $C_0w_{b,\alpha,q}(0,T)^{1/q}T^{\gamma_0}N^{-\gamma_0}_0\leq 1/2$. For $S_0=t_i$ and $T_0=t_{i+1}$ with $i=0,1,\ldots,N_0-1$, one gets from (\ref{5.8}) that
\begin{eqnarray}\label{5.9}
\begin{split} [\hat{X}^{(n-1,n)}]_{\mathcal{C}_{[t_i,t_{i+1}],p}^{\gamma_0,w_{b,\alpha,q}^{1/q}}}
 \leq & \sum_{j=1}^{n-1}C_02^{-j+1} \|\hat{X}^{(n-j-1,n-j)}_{t_i}\|_{L^p(\Omega)} + 2^{-n+1}[\hat{X}^{(0,1)}]_{\mathcal{C}_{[t_i,t_{i+1}],p}^{\gamma_0,w_{b,\alpha,q}^{1/q}}}\\ \leq & \sum_{j=1}^{n-1}C_02^{-j+1} \|\hat{X}^{(n-j-1,n-j)}_{t_i}\|_{L^p(\Omega)} + 2^{-n+1}T^{\frac{1-\alpha}{2}},
\end{split}
\end{eqnarray}
where in the last inequality we have used $[\hat{X}^{(0,1)}]_{\mathcal{C}_{[0,T],p}^{\gamma_0,w_{b,\alpha,q}^{1/q}}} \leq  T^{\frac{1-\alpha}{2}}$ since
\begin{eqnarray*}
 \|\hat{X}^{(0,1)}_{s,t}\|_{L^p(\Omega)}=\Big\|\int_s^tb(r,x+B_r)dr\Big\|_{L^p(\Omega)}
 \leq  w_{b,\alpha,q}(s,t)^{\frac{1}{q}}(t-s)^{\gamma_0}T^{\frac{1-\alpha}{2}}.
\end{eqnarray*}
(\ref{5.9}) combined with (\ref{5.7}) yields
\begin{eqnarray}\label{5.10}
\begin{split} \|\hat{X}^{(n-1,n)}\|_{p,[t_i,t_{i+1}]}\leq & \|\hat{X}^{(n-1,n)}_{t_i}\|_{L^p(\Omega)}+2^{-1}\|\hat{X}^{(n-2,n-1)}_{t_i}\|_{L^p(\Omega)}+2^{-2}
[\hat{X}^{(n-2,n-1)}]_{\mathcal{C}_{[t_i,t_{i+1}],p}^{\gamma_0,w_{b,\alpha,q}^{1/q}}}\\  \leq & \sum_{j=1}^n2^{-j+1}\|\hat{X}^{(n-j,n-j+1)}_{t_i}\|_{L^p(\Omega)} +  2^{-n}T^{\frac{1-\alpha}{2}}.
\end{split} \end{eqnarray}
By (\ref{5.9}), (\ref{5.10}) and induction, for every $n>m\geq 2$, we conclude
\begin{eqnarray}\label{5.11}
\begin{split}  [\hat{X}^{(m,n)}]_{\mathcal{C}_{[t_i,t_{i+1}],p}^{\gamma_0,w_{b,\alpha,q}^{1/q}}}
  \leq& \sum_{k=m}^{n-1}\sum_{j=1}^{k}C_02^{-j+1} \|\hat{X}^{(k-j,k-j+1)}_{t_i}\|_{L^p(\Omega)}+ \sum_{k=m}^{n-1}2^{-k}T^{\frac{1-\alpha}{2}}\\ \leq& \sum_{k=m}^{n-1}\sum_{j=1}^{k}C_02^{-j+1} \|\hat{X}^{(k-j,k-j+1)}_{t_i}\|_{L^p(\Omega)} + 2^{-m+1}T^{\frac{1-\alpha}{2}}
\end{split}
\end{eqnarray}
and
\begin{eqnarray}\label{5.12}
\begin{split} \|\hat{X}^{(m,n)}\|_{p,[t_i,t_{i+1}]} \leq & \sum_{k=m}^{n-1}\sum_{j=1}^{k+1}2^{-j+1}\|\hat{X}^{(k-j+1,k-j+2)}_{t_i}\|_{L^p(\Omega)}+  \sum_{k=m}^{n-1}2^{-k-1}T^{\frac{1-\alpha}{2}}
\\
\leq & \sum_{k=m}^{n-1}\sum_{j=1}^{k+1}2^{-j+1}\|\hat{X}^{(k-j+1,k-j+2)}_{t_i}\|_{L^p(\Omega)} +  2^{-m}T^{\frac{1-\alpha}{2}}.
\end{split} \end{eqnarray}

For $i=0$, combined with (\ref{5.11}) and (\ref{5.12}), we derive
\begin{eqnarray}\label{5.13}
\begin{split} \|\hat{X}^{(m,n)}_{s,t}\|_{L^p(\Omega)}\leq& w_{b,\alpha,q}(s,t)^{\frac{1}{q}}
(t-s)^{\gamma_0}[\hat{X}^{(m,n)}]_{\mathcal{C}_{[t_0,t_1],p}^{\gamma_0,w_{b,\alpha,q}^{1/q}}}
\\ \leq& w_{b,\alpha,q}(s,t)^{\frac{1}{q}}(t-s)^{\gamma_0}2^{-m+1}T^{\frac{1-\alpha}{2}}, \quad \forall \ (s,t)\in [t_0,t_1]_\leq^2,
\end{split}
\end{eqnarray}
and
\begin{eqnarray}\label{5.14}
\| \hat{X}^{(m,n)}\|_{p,[t_0,t_1]} \leq  2^{-m}T^{\frac{1-\alpha}{2}}.
\end{eqnarray}

We choose $p$ large enough such that $p\gamma_0>1$. By (\ref{5.13}) and Lemma \ref{lem2.5}, $\{X_t^{(m,n)}\}_{t\in [0,T]}$ has a continuous realization (still denoted by itself) which is $\gamma$-H\"{o}lder continuous in $t$ for every $\gamma\in (0,\gamma_0)$. Moreover, by (\ref{2.17}) and H\"{o}lder's inequality, for $\gamma\in (0,\gamma_0)$ and $p\in [2,\infty)$, there exists a constant $C_1(\gamma_0,\gamma,p)$ such that
\begin{eqnarray*}
[\hat{X}^{(m,n)}]_{\mathcal{C}_{p,[t_0,t_1]}^{\gamma,w_{b,\alpha,q}^{1/q}}}\leq C_12^{-m+1}T^{\frac{1-\alpha}{2}}.
\end{eqnarray*}
This, coupled with (\ref{5.14}), results in
\begin{eqnarray*}
\|\hat{X}^{(m,n)}\|_{p,\gamma,q,w_{b,\alpha,q},[t_0,t_1]}
\leq [1+2C_1]2^{-m}T^{\frac{1-\alpha}{2}},
\end{eqnarray*}
where the notation $\|\cdot\|_{p,\gamma,q,w_{b,\alpha,q},t_1}$ is given in (\ref{1.14}). Consequently, $\{\hat{X}^{(n)}_\cdot\}_{n\geq 1}$ is a Cauchy sequence in the Banach space $\mathcal{C}_{p,[t_0,t_1]}^{\gamma,w_{b,\alpha,q}^{1/q}}$ and there exists a stochastic process $\tilde{X}\in \mathcal{C}_{p,[t_0,t_1]}^{\gamma,w_{b,\alpha,q}^{1/q}}$ such that
\begin{eqnarray}\label{5.15}
\lim_{n\rightarrow \infty}\|\tilde{X}-\hat{X}^{(n)}\|_{p,\gamma,q,w_{b,\alpha,q},[t_0,t_1]}
=0.
\end{eqnarray}
By (\ref{5.15}), if one lets $n$ approach to zero in (\ref{1.18}), then $\tilde{X}$ satisfies (\ref{1.1}) on $[t_0,t_1]$. By the uniqueness of strong solution of (\ref{1.1}), we complete the proof on the time interval $[t_0,t_1]$.

We then choose $i=1$, and by (\ref{5.11}), (\ref{5.12}) and (\ref{5.14}) to get
\begin{eqnarray}\label{5.16}
\begin{split}  [\hat{X}^{(m,n)}]_{\mathcal{C}_{[t_1,t_2],p}^{\gamma_0,w_{b,\alpha,q}^{1/q}}}
  \leq& \sum_{k=m}^{n-1}\sum_{j=1}^{k}C_02^{-j+1} \|\hat{X}^{(k-j,k-j+1)}_{t_1}\|_{L^p(\Omega)} + 2^{-m+1}T^{\frac{1-\alpha}{2}}  \\ \leq& \sum_{k=m}^{n-1}k 2^{-k+1}T^{\frac{1-\alpha}{2}} + 2^{-m+1}T^{\frac{1-\alpha}{2}}
\end{split}
\end{eqnarray}
and
\begin{eqnarray}\label{5.17}
\begin{split} \|\hat{X}^{(m,n)}\|_{p,[t_1,t_2]}
\leq & \sum_{k=m}^{n-1}\sum_{j=1}^{k+1}2^{-j+1}\|\hat{X}^{(k-j+1,k-j+2)}_{t_1}\|_{L^p(\Omega)} +  2^{-m}T^{\frac{1-\alpha}{2}} \\ \leq& \sum_{k=m}^{n-1}(k+1) 2^{-k}T^{\frac{1-\alpha}{2}} + 2^{-m}T^{\frac{1-\alpha}{2}}.
\end{split} \end{eqnarray}
For every $2<m_0\in \mathbb{N}$, then
\begin{eqnarray}\label{5.18}
\begin{split} \sum_{k=m_0}^{\infty}k2^{-k+1}\leq& 2\int_{m_0}^\infty r2^{-r+1}dr=2\frac{d}{d\tau}\Big[ \int_{m_0}^\infty \tau^{r}dr\Big]\Big|_{\tau=\frac{1}{2}}\\ =&2^{-m_0+2}\big[ m_0\log(2)^{-1}+\log(2)^{-2}\big]\leq (m_0+1)2^{-m_0+3}.
\end{split}
\end{eqnarray}
In view of (\ref{5.16})--(\ref{5.18}), it follows that
\begin{eqnarray}\label{5.19}
\begin{split}  \|\hat{X}^{(m,n)}_{s,t}\|_{L^p(\Omega)}
\leq& w_{b,\alpha,q}(s,t)^{\frac{1}{q}}(t-s)^{\gamma_0}
 [\hat{X}^{(m,n)}]_{\mathcal{C}_{[t_1,t_2],p}^{\gamma_0,w_{b,\alpha,q}^{1/q}}}
\\  \leq& w_{b,\alpha,q}(s,t)^{\frac{1}{q}}(t-s)^{\gamma_0} (m+2)2^{-m+3}T^{\frac{1-\alpha}{2}}, \quad \forall \ (s,t)\in [t_1,t_2]_\leq^2,
\end{split}
\end{eqnarray}
and
\begin{eqnarray}\label{5.20}
\begin{split} \|\hat{X}^{(m,n)}\|_{p,[t_1,t_2]}\leq (m+2)2^{-m+3}T^{\frac{1-\alpha}{2}}.
\end{split} \end{eqnarray}
According to (\ref{5.19}) and (\ref{5.20}), $\{\hat{X}^{(n)}_\cdot\}_{n\geq 1}$ is also a Cauchy sequence in the Banach space $\mathcal{C}_{p,[t_1,t_2]}^{\gamma,w_{b,\alpha,q}^{1/q}}$. Applying the same argument as in the time interval $[t_0,t_1]$, one has
\begin{eqnarray*}
\lim_{n\rightarrow \infty}\|X-\hat{X}^{(n)}\|_{p,\gamma,q,w_{b,\alpha,q},[t_1,t_2]}
=0.
\end{eqnarray*}
Iterating this process finitely many times and noting the following fact
\begin{eqnarray*}
\lim_{m\rightarrow \infty}m^i2^{-m}=0, \ \ i=2,3,\ldots,N_0-1,
\end{eqnarray*}
we obtain the uniform estimate (\ref{1.19}) over the time interval $[0,T]$.




\end{document}